\documentclass[11pt]{article}
\usepackage{amsmath,amssymb}
\usepackage{graphicx}
\usepackage{amsfonts}

\usepackage{amsthm}
\newtheorem{thm}{Theorem}[section]

\usepackage[USenglish]{babel} 
\usepackage[T1]{fontenc}
\usepackage[ansinew]{inputenc}
\usepackage{mathrsfs}
\usepackage{stmaryrd} 
\usepackage{epstopdf} 
\usepackage{verbatim}
\usepackage{MnSymbol}  
\usepackage{accents}   
\usepackage{bbm}			 
\usepackage{amsbsy}		 
\usepackage{color}     
\usepackage{authblk}   

\numberwithin{equation}{section}
\numberwithin{figure}{section}

\DeclareMathOperator{\sech}{sech}

\title{Data-driven structure-preserving model reduction for stochastic Hamiltonian systems}
\date{}
\author[1,2,3]{Tomasz M. Tyranowski\thanks{\texttt{tomasz.tyranowski@ipp.mpg.de}}}

\affil[1]{\small University of Twente, Department of Applied Mathematics \authorcr PO Box 217, 7500AE Enschede, The Netherlands}
\affil[2]{\small Max-Planck-Institut f\"ur Plasmaphysik \authorcr Boltzmannstra{\ss}e 2, 85748 Garching, Germany}
\affil[3]{\small Technische Universit\"{a}t M\"{u}nchen, Zentrum Mathematik \authorcr Boltzmannstra{\ss}e 3, 85748 Garching, Germany}

\begin{document}

\maketitle

\begin{abstract}
In this work we demonstrate that SVD-based model reduction techniques known for ordinary differential equations, such as the proper orthogonal decomposition, can be extended to stochastic differential equations in order to reduce the computational cost arising from both the high dimension of the considered stochastic system and the large number of independent Monte Carlo runs. We also extend the proper symplectic decomposition method to stochastic Hamiltonian systems, both with and without external forcing, and argue that preserving the underlying symplectic or variational structures results in more accurate and stable solutions that conserve energy better than when the non-geometric approach is used. We validate our proposed techniques with numerical experiments for a semi-discretization of the stochastic nonlinear Schr\"{o}dinger equation and the Kubo oscillator.
\end{abstract}

\section{Introduction}
\label{sec:intro}
The purpose of this work is twofold: to demonstrate that the conventional SVD-based model reduction methods for ordinary differential equations (ODEs) can be extended to stochastic differential equations (SDEs), and to show that in the case of stochastic Hamiltonian systems it is beneficial to maintain their symplectic (or variational) structure in the construction of the reduced spaces by extending the existing structure-preserving model reduction techniques for deterministic Hamiltonian systems. 

Model reduction methods have been introduced in order to reduce the computational cost of solving high-dimensional dynamical systems. The goal of these techniques is to construct lower-dimensional models which are less expensive to solve numerically, but still capture the dominant features of the dynamics of the original system (see \cite{Antoulas2001}, \cite{BennerGugercin2015}, \cite{BennerBook2017}, \cite{QuarteroniBook2015} and the references therein). The Proper Orthogonal Decomposition (POD), first introduced in \cite{Lumley1967}, is an SVD-based data-driven model reduction technique that employs an offline-online splitting. In the offline stage the available empirical data about the solutions of the full system are used to identify an optimal subspace of the full configuration space. The equations governing the evolution of the system are then projected to that subspace, and in the online stage such a reduced model is solved numerically at a lower computational cost. The POD method has proved very successful and has been used in many scientific and engineering problems, such as fluid dynamics (\cite{HolmesRowleyBook}, \cite{ItoRavindran1998}, \cite{ItoRavindran2001}, \cite{Kunisch2002}, \cite{Lassila2014}, \cite{Peterson1989}, \cite{RowleyMurray2004}), electric circuit analysis (\cite{ParriloMarsden1999}), or structural dynamics (\cite{Amabili2003}); see also \cite{Buchfink2019}, \cite{Chaturantabut2010}, \cite{Drohmann2012}, \cite{GlavaskiMarsden1998}, \cite{RathinamPetzold2003}.
 
Consider a general stochastic differential equation

\begin{equation}
\label{eq: Stochastic differential equation - general}
d_t u = a(u)\,dt + \sum_{\nu=1}^{m} b_\nu(u)\circ dW^\nu(t),
\end{equation}

\noindent
where $a:\mathbb{R}^n \longrightarrow \mathbb{R}^n$ is the drift function, $[b_1, \ldots, b_m]:\mathbb{R}^n\longrightarrow \mathbb{R}^{n \times m}$ is the diffusion matrix, $W(t)=(W^1(t),\ldots,W^m(t))$ is the standard $m$-dimensional Wiener process, and $\circ$ denotes Stratonovich integration. We use $d_t$ to denote the stochastic differential of stochastic processes (other than the Wiener process $W(t)$) to avoid confusion with the exterior derivative $d$ of differential forms. We will assume that the drift function and the diffusion matrix are sufficiently smooth and satisfy all the necessary conditions for the existence and uniqueness of solutions to \eqref{eq: Stochastic differential equation - general} (see \cite{ArnoldSDE}, \cite{IkedaWatanabe1989}, \cite{KloedenPlatenSDE}, \cite{Kunita1997}). Much like in the deterministic case, numerical simulations of \eqref{eq: Stochastic differential equation - general} become computationally expensive when the dimension $n$ of the stochastic process $u(t)$ is large. Such a situation may occur, for instance, when \eqref{eq: Stochastic differential equation - general} comes from a semi-discretization of a stochastic partial differential equation (see Section~\ref{sec: Stochastic Nonlinear Schrodinger Equation}). Also, in order to calculate the statistical properties of the stochastic process $u(t)$, such as the probability density function or the expected value, one typically needs to simulate \eqref{eq: Stochastic differential equation - general} for a very large number of sample paths. We will show that the computational cost of both these tasks can be alleviated by adapting the POD method to the stochastic setting. A number of model reduction methods have been applied to stochastic systems in various contexts. Examples include partial differential equations with random coefficients (\cite{BoyavalBris2010}, \cite{Boyaval2009}, \cite{Doostan2007}, \cite{Ghanem2007}), parametric closure models for the stochastic Burgers' equation (\cite{Lu2020}), or a variance reduction method for SDEs (\cite{BoyavalLelievre2010}). Stochastic reduced models were also used as noisy perturbations of deterministic reduced models in order to account for unresolved small-scale features (\cite{Bastine2018}, \cite{Chorin2015}, \cite{Lu2017}, \cite{Paskyabi2020}). However, to the best of our knowledge the adaptation of the POD method to SDEs driven by a Wiener process has been much less investigated. An application of the POD method which is similar in spirit to our approach appears in \cite{Iliescu2018} and \cite{XieBao2018}, but only in the specific context of solving the stochastic Burgers equation, empirical approximation of the nonlinear term is not addressed, and only low-order integration in time is used (see also \cite{Burkardt2007}).

The POD method can, in principle, be applied to a Hamiltonian system. There is, however, no guarantee that the reduced system will maintain the Hamiltonian structure, nor that it will be stable, which may result in a blow-up of its solutions (see \cite{Prajna2003}, \cite{RathinamPetzold2003}). A model reduction technique that retains the symplectic structure of Hamiltonian systems was introduced in \cite{PengMohseni2016}. In analogy to POD, this method is called the Proper Symplectic Decomposition (PSD). The PSD method has been proven to preserve the energy and stability of the system, and to yield better numerical solutions, especially when combined with symplectic integration in time. Therefore, the PSD method is better suited for model reduction of Hamiltonian systems than the classical POD approach, especially when long-time integration is required; see also \cite{AfkhamHesthaven2017}, \cite{AfkhamHesthaven2018}, \cite{AfkhamHesthaven2019}, \cite{Carlberg2015}, \cite{Chaturantabut2016}, \cite{Gugercin2012}, \cite{HesthavenReview2021}, \cite{Karasozen2018}, \cite{LallMarsden2003}, \cite{PengMohseniProceedings2016}, \cite{PengMohseniPreprint2016}, \cite{Polyuga2010}.

A stochastic Hamiltonian system is an SDE of the form

\begin{align}
\label{eq: Stochastic Hamiltonian system - general}
d_t q = \frac{\partial H}{\partial p}\,dt + \sum_{\nu=1}^{m} \frac{\partial h_\nu}{\partial p}\circ dW^\nu(t), \qquad\qquad d_t p = -\frac{\partial H}{\partial q}\,dt - \sum_{\nu=1}^{m} \frac{\partial h_\nu}{\partial q}\circ dW^\nu(t),
\end{align}

\noindent
where $H:\mathbb{R}^n \times \mathbb{R}^n\longrightarrow \mathbb{R}$ and $h_\nu:\mathbb{R}^n \times \mathbb{R}^n\longrightarrow \mathbb{R}$ for $\nu=1,\ldots,m$ are the Hamiltonian functions. Such systems can be used to model, e.g., mechanical systems with uncertainty, or error, assumed to arise from random forcing, limited precision of experimental measurements, or unresolved physical processes on which the Hamiltonian of the deterministic system might otherwise depend. Particular examples include modeling synchrotron oscillations of particles in particle storage rings (see \cite{DomeAccelerators}, \cite{SeesselbergParticleStorageRings}) and stochastic dynamics of the interactions of singular solutions of the EPDiff basic fluids equation (see \cite{HolmTyranowskiSolitons}, \cite{HolmTyranowskiVirasoro}). More examples are discussed in Section~\ref{sec: Numerical experiments}; see also \cite{HolmTyranowskiGalerkin}, \cite{LelievreStoltz2016}, \cite{Mao2007}, \cite{Nelson1988}, \cite{SanzSerna1999}, \cite{Shardlow2003}, \cite{Soize1994}, \cite{Talay2002}. Similar to their deterministic counterparts, stochastic Hamiltonian systems possess several important geometric features. In particular, their phase space flows (almost surely) preserve the canonical symplectic structure. We will argue that maintaining this property in model reduction is beneficial because the resulting reduced systems can then be integrated using stochastic symplectic methods (see \cite{Anmarkrud2017}, \cite{AntonWeak2014}, \cite{AntonGlobalError2013}, \cite{Anton2013}, \cite{BrehierCohen2021}, \cite{Burrage2012}, \cite{Burrage2014}, \cite{ChenCohen2020}, \cite{Cohen2014}, \cite{CohenVilmart2021}, \cite{Cordoni2022}, \cite{AntonHighOrder2014}, \cite{HolmTyranowskiSolitons}, \cite{HolmTyranowskiGalerkin}, \cite{HongSunBook2023}, \cite{Hong2015}, \cite{KrausTyranowski2019}, \cite{MaDing2012}, \cite{MaDing2015}, \cite{MilsteinRepin2001}, \cite{MilsteinRepin}, \cite{Misawa2010}, \cite{SunWang2016}, \cite{WangPHD}, \cite{Wang2014}, \cite{Wang2017}, \cite{Zhou2017}), and consequently more accurate numerical solutions can be obtained. This goal can be achieved by an appropriate adaptation of the PSD method in the stochastic setting. We are not aware of any previous studies on this topic.

\paragraph{Main content}
The main content of the remainder of this paper is, as follows. 
\begin{description}
\item
In Section~\ref{sec: Model reduction for SDEs} we briefly review the POD method and present how it can be applied to the general SDE \eqref{eq: Stochastic differential equation - general} in computations involving a single realization of the Wiener process.
\item
In Section~\ref{sec: Model reduction for stochastic Hamiltonian systems} we briefly review the PSD method and present how it can be applied to the stochastic Hamiltonian system \eqref{eq: Stochastic Hamiltonian system - general} in computations involving a single realization of the Wiener process. We also further extend the PSD method to stochastic forced Hamiltonian systems.
\item
In Section~\ref{sec: Reduction of the number of Monte Carlo runs} we discuss how the problem of calculating the statistical properties of solutions to \eqref{eq: Stochastic differential equation - general} or \eqref{eq: Stochastic Hamiltonian system - general} can be recast as the problems presented in Section~\ref{sec: Model reduction for SDEs} and Section~\ref{sec: Model reduction for stochastic Hamiltonian systems}, respectively.
\item
In Section~\ref{sec: Numerical experiments} we present the results of our numerical experiments for the stochastic Nonlinear Schr\"{o}dinger Equation and the Kubo oscillator, both with and without external forcing.
\item
Section~\ref{sec: Summary} contains the summary of our work.
\end{description}

\section{Model reduction for SDEs}
\label{sec: Model reduction for SDEs}

The Proper Orthogonal Decomposition is one of the standard model reduction techniques for ordinary differential equations (see \cite{Antoulas2001}, \cite{BennerGugercin2015}, \cite{BennerBook2017}, \cite{QuarteroniBook2015}). In this section we demonstrate how the POD method can be adapted in the context of stochastic differential equations.

\subsection{Proper Orthogonal Decomposition}
\label{sec: Proper Orthogonal Decomposition}

Suppose we would like to solve \eqref{eq: Stochastic differential equation - general} for a single realization of the Wiener process $W(t)$. If the dimension $n$ of the stochastic process $u(t)$ is a very high number, then the system \eqref{eq: Stochastic differential equation - general} becomes very expensive to solve numerically. The main idea of model reduction is to approximate such a high-dimensional stochastic dynamical system using a lower-dimensional one that can capture the dominant dynamic properties. Let $\Delta$ be an $n \times r$ matrix representing empirical data on the system~\eqref{eq: Stochastic differential equation - general}. For instance, $\Delta$ can be a collection of snapshots of a solution of this system for the given realization of the Wiener process, 

\begin{equation}
\label{eq: Empirical data for POD}
\Delta = [u(t_1) \, u(t_2) \, \ldots \, u(t_r)],
\end{equation}

\noindent
at times $t_1,\ldots,t_r$. These snapshots are calculated for a particular set of initial conditions or values of parameters that the system~\eqref{eq: Stochastic differential equation - general} depends on (see Section~\ref{sec: Numerical experiments}). A low-rank approximation of $\Delta$ can be done by performing the singular value decomposition (SVD) of $\Delta$ and truncating it after the first $k$ largest singular values, that is,

\begin{equation}
\label{eq: SVD of Delta}
\Delta = U \Sigma V^T \approx U_k \Sigma_k V_k^T,
\end{equation}

\noindent
where $\Sigma = \text{diag}(\sigma_1, \sigma_2, \ldots)$ is the diagonal matrix of the singular values, $U$ and $V$ are orthogonal matrices, $\Sigma_k$ is the diagonal matrix of the first $k$ largest singular values, and $U_k$ and $V_k$ are orthogonal matrices constructed by taking the first $k$ columns of $U$ and $V$, respectively. Let $\xi$ denote a vector in $\mathbb{R}^k$. Substituting $u=U_k\xi$ in \eqref{eq: Stochastic differential equation - general} yields a reduced SDE for $\xi(t)$ as

\begin{equation}
\label{eq: POD reduced stochastic differential equation}
d_t \xi = U^T_ka(U_k\xi)\,dt + \sum_{\nu=1}^{m} U^T_kb_\nu(U_k\xi)\circ dW^\nu(t).
\end{equation}

\noindent
The corresponding initial condition is calculated as $\xi_0 = U_k^T u_0$, where $u_0$ is an initial condition for~\eqref{eq: Stochastic differential equation - general}. If the singular values of $\Delta$ decay sufficiently fast, then one can obtain a good approximation of $\Delta$ for $k$ such that $k\ll n$. Equation~\eqref{eq: POD reduced stochastic differential equation} is then a low-dimensional approximation of \eqref{eq: Stochastic differential equation - general} and can be solved more efficiently. The approximate solution of \eqref{eq: Stochastic differential equation - general} is then reconstructed as $u(t)=U_k\xi(t)$. The process of constructing the matrix $U_k$ from the empirical/simulation data ensemble $\Delta$ is typically called the \emph{offline} stage of model reduction. Depending on the size of the data, this stage can be computationally expensive, but it is performed only once. Solving the low-dimensional system \eqref{eq: POD reduced stochastic differential equation} is usually referred to as the \emph{online} stage of model reduction, and is supposed to be faster and more efficient than solving the full system \eqref{eq: Stochastic differential equation - general}.

\subsection{Discrete Empirical Interpolation Method}
\label{sec: Discrete Empirical Interpolation Method}

If the drift $a$ or any of the diffusion terms $b_\nu$ is a complicated nonlinear function, then solving the reduced system \eqref{eq: POD reduced stochastic differential equation} may not bring any computational savings, because one usually needs to compute the state variable $u=U_k \xi$ in the original coordinate system, evaluate the nonlinear drift or diffusion terms, and then project back to the column space of $U_k$. A technique called the Discrete Empirical Interpolation Method (DEIM) has been developed for ODEs in order to reduce the complexity in evaluating the nonlinear terms (see \cite{Chaturantabut2010}, \cite{RathinamPetzold2003}). This technique can be readily adapted also in the stochastic setting. For completeness and for the benefit of the reader, below we briefly outline the main ideas of DEIM. The further details can be found in, e.g., \cite{Chaturantabut2010}, \cite{PengMohseni2016}. Let the drift function be expressed as

\begin{equation}
\label{eq: Drift function - decomposition}
a(u) = L u + a_N(u),
\end{equation}

\noindent
where $L$ is an $n \times n$ matrix, and $a_N$ represents the nonlinear part of $a$. First, suppose that $a_N(u)$ lies approximately in the range of an $n \times \bar k$ matrix $\Psi$. Similar to $U_k$ in \eqref{eq: SVD of Delta}, the matrix $\Psi$ can be found by performing the SVD of another empirical data ensemble, namely

\begin{equation}
\label{eq: Empirical data for DEIM}
\bar \Delta = [a_N(u(t_1)) \,\, a_N(u(t_2)) \,\, \ldots \,\, a_N(u(t_r))],
\end{equation}

\noindent
and truncating it after the first $\bar k$ largest singular values. Next, calculate only $\bar k$ components of the vector $a_N(u)$ with the preselected indices $\beta_1, \ldots \beta_{\bar k}$. Those $\bar k$ components of $a_N(u)$ can be conveniently denoted by the expression $P^T a_N(u)$, where the $n \times \bar k$ matrix $P$ is defined as

\begin{equation}
\label{eq: P matrix}
P = [\mathbbm{e}_{\beta_1}, \ldots, \mathbbm{e}_{\beta_{\bar k}}],
\end{equation}

\noindent
and $\mathbbm{e}_{\beta_i}$ denotes the $\beta_i$-th column of the identity matrix $\mathbb{I}_n$. Given $\Psi$, the suitable set of indices $\beta_1, \ldots \beta_{\bar k}$ can be inductively constructed using a greedy algorithm (see Algorithm 1 in \cite{Chaturantabut2010} for details). The DEIM approximation of the nonlinear term is then expressed as

\begin{equation}
\label{eq: DEIM approximation of f_N}
\bar a_N(u) = \Psi (P^T \Psi)^{-1} P^T a_N(u),
\end{equation}

\noindent
and the approximation of the drift term in \eqref{eq: POD reduced stochastic differential equation} becomes

\begin{equation}
\label{eq: DEIM approximation of the drift term}
U^T_ka(U_k\xi) \approx \bar L \xi + W g(\xi),
\end{equation}

\noindent
where

\begin{equation}
\label{eq: L bar, W, and g(xi) for DEIM}
\bar L = U^T_k L U_k, \qquad\qquad W = U^T_k \Psi (P^T \Psi)^{-1}, \qquad\qquad g(\xi) = P^T a_N(U_k \xi).
\end{equation}

\noindent
Note that the matrices $\bar L$ and $W$ are calculated only once at the offline stage, and at the online stage the function $g(\xi)$ evaluates only $\bar k$ components of $a_N(U_k \xi)$. A similar DEIM approximation can be applied to each of the diffusion terms $b_\nu(u)$, thus reducing the computational complexity of the reduced system \eqref{eq: POD reduced stochastic differential equation}.

\subsection{Time integration}
\label{sec: POD Time integration}

The SDEs \eqref{eq: Stochastic differential equation - general} and \eqref{eq: POD reduced stochastic differential equation} can be solved numerically using any general purpose stochastic numerical schemes (see \cite{KloedenPlatenSDE}, \cite{MilsteinBook} and the references therein). In this work we will focus our attention on several stochastic Runge-Kutta methods, namely the explicit stochastic Heun and $R2$ methods (see \cite{Burrage1996}, \cite{Burrage1998}, \cite{Burrage2000}, \cite{BurragePhDThesis}, \cite{KloedenPlatenSDE}), and the implicit stochastic midpoint method (see \cite{HolmTyranowskiGalerkin}, \cite{KrausTyranowski2019}, \cite{MaDing2015}, \cite{MilsteinRepin}). The stochastic mipoint method for \eqref{eq: Stochastic differential equation - general} takes the form

\begin{equation}
\label{eq: Stochastic midpoint method}
u_{i+1} = u_i + a\bigg( \frac{u_i+u_{i+1}}{2}\bigg) \Delta t + \sum_{\nu=1}^m b_\nu\bigg( \frac{u_i+u_{i+1}}{2}\bigg) \Delta W^\nu,
\end{equation}

\noindent
where $\Delta t$ denotes the time step and $\Delta W = (\Delta W^1, \ldots, \Delta W^m)$ are the increments of the Wiener process. All the mentioned methods are strongly convergent of order 1 for systems driven by a commutative noise, and of order 1/2 in the non-commutative case.

\section{Model reduction for stochastic Hamiltonian systems}
\label{sec: Model reduction for stochastic Hamiltonian systems}

The Proper Symplectic Decomposition is a model reduction technique that has been developed for deterministic Hamiltonian systems (see \cite{AfkhamHesthaven2017}, \cite{PengMohseni2016}). In this section we demonstrate how the PSD method can be adapted in the context of stochastic Hamiltonian systems.

\subsection{Proper Symplectic Decomposition}
\label{sec: Proper Symplectic Decomposition}

The stochastic Hamiltonian system \eqref{eq: Stochastic Hamiltonian system - general} can be equivalently written as

\begin{equation}
\label{eq: Stochastic Hamiltonian system in terms of J}
d_t u = \mathbb{J}_{2n} \nabla_u H(u) \, dt + \sum_{\nu=1}^m \mathbb{J}_{2n} \nabla_u h_\nu(u) \circ dW^\nu(t),
\end{equation}

\noindent 
where $u=(q,p)$ and $\mathbb{J}_{2n}$ is the canonical symplectic matrix defined as

\begin{equation}
\label{eq: Canonical symplectic matrix}
\mathbb{J}_{2n}=\left(\begin{matrix}
0 & \mathbb{I}_{n} \\
-\mathbb{I}_{n} & 0
\end{matrix}\right),
\end{equation}

\noindent
with $\mathbb{I}_{n}$ denoting the $n \times n$ identity matrix. In a manner similar to its deterministic counterpart (see, e.g., \cite{HLWGeometric}, \cite{HolmGMS}, \cite{MarsdenRatiuSymmetry}), the stochastic Hamiltonian system \eqref{eq: Stochastic Hamiltonian system in terms of J} possesses several characteristic properties. Its stochastic flow $F_t$ (almost surely) preserves the canonical symplectic form $\Omega = \sum_{i=1}^n dq^i \wedge dp^i$ on the phase space $\mathbb{R}^n \times \mathbb{R}^n$. This property expressed in terms of the standard basis for $\mathbb{R}^{2n}$ takes the form of the condition

\begin{equation}
\label{eq: Symplecticity of the flow}
(DF_t)^T \mathbb{J}_{2n} DF_t = \mathbb{J}_{2n},
\end{equation} 

\noindent
where $DF_t$ denotes the Jacobi matrix of the flow map $F_t$. Moreover, if the Hamiltonian function $H$ commutes with all the Hamiltonian functions $h_\nu$, that is, if the canonical Poisson bracket satisfies

\begin{equation}
\label{eq: Poisson bracket of the Hamiltonians}
\{H, h_\nu\} = (\nabla_u H)^T \mathbb{J}_{2n} \nabla_u h_\nu = \sum_{i=1}^n \bigg( \frac{\partial H}{\partial q^i} \frac{\partial h_\nu}{\partial p^i} - \frac{\partial H}{\partial p^i} \frac{\partial h_\nu}{\partial q^i} \bigg) = 0
\end{equation}

\noindent
for all $\nu=1,\ldots, m$, then the flow $F_t$ also preserves (almost surely) the Hamiltonian function $H$, which can be easily verified by calculating the stochastic differential

\begin{equation}
\label{eq: Stochastic differential of H}
d_tH(u(t)) = \sum_{\nu=1}^m \{H,h_\nu\}\circ dW^\nu(t),
\end{equation}

\noindent
where $u(t)=F_t(u_0)$ is the solution of \eqref{eq: Stochastic Hamiltonian system in terms of J} with the initial condition $u(0)=u_0$, and we used the rules of Stratonovich calculus (see \cite{Bismut}, \cite{HolmTyranowskiGalerkin}, \cite{LaCa-Or2008}, \cite{MilsteinRepin}).

Suppose we would like to solve \eqref{eq: Stochastic Hamiltonian system in terms of J} for a single realization of the Wiener process. Again, if the dimension $2n$ of the stochastic process $u(t)$ is very high, then the system \eqref{eq: Stochastic Hamiltonian system in terms of J} becomes expensive to solve. In principle, the POD method described in Section~\ref{sec: Model reduction for SDEs} could be applied to \eqref{eq: Stochastic Hamiltonian system in terms of J}, but there is no guarantee that the reduced system \eqref{eq: POD reduced stochastic differential equation} will retain the Hamiltonian structure and the geometric properties of \eqref{eq: Stochastic Hamiltonian system in terms of J}. The PSD method, first proposed in \cite{PengMohseni2016} for deterministic systems, constructs a reduced model that is also a Hamiltonian system. A $2n \times 2k$ matrix is called symplectic if it satisfies the condition

\begin{equation}
\label{eq: Symplectic matrix}
A^T \mathbb{J}_{2n} A = \mathbb{J}_{2k}.
\end{equation}

\noindent
For a symplectic matrix $A$, we can define its symplectic inverse $A^+ = \mathbb{J}_{2k}^T A^T \mathbb{J}_{2n}$. It is an inverse in the sense that $A^+ A = \mathbb{I}_{2k}$. Let $\xi$ be a vector in $\mathbb{R}^{2k}$. Substituting $u = A\xi$ in \eqref{eq: Stochastic Hamiltonian system in terms of J} yields a reduced equation

\begin{align}
\label{eq: Reduced stochastic Hamiltonian system}
d_t \xi &= A^+ \mathbb{J}_{2n} \nabla_u H(u) \, dt + \sum_{\nu=1}^m A^+\mathbb{J}_{2n} \nabla_u h_\nu(u) \circ dW^\nu(t) \nonumber \\
        &= \mathbb{J}_{2k}\nabla_\xi H(A\xi) \, dt + \sum_{\nu=1}^m \mathbb{J}_{2k} \nabla_\xi h_\nu(A \xi) \circ dW^\nu(t),
\end{align}

\noindent
which is a lower-dimensional stochastic Hamiltonian system with the Hamiltonian functions $\tilde H(\xi) = H(A \xi)$ and $\tilde h_\nu(\xi) = h_\nu(A \xi)$ for $\nu=1,\ldots,m$. The corresponding initial condition is calculated as $\xi_0 = A^+ u_0$, where $u_0$ is an initial condition for~\eqref{eq: Stochastic Hamiltonian system in terms of J}. Given a set of empirical data on a Hamiltonian system, the PSD method constructs a symplectic matrix $A$ which best approximates that data in a lower-dimensional subspace. Several algorithms have been proposed to construct $A$, namely the cotangent lift algorithm, the complex SVD algorithm, the greedy algorithm, and the nonlinear programming algorithm (see \cite{AfkhamHesthaven2017}, \cite{PengMohseni2016}). Any of these algorithms can be adapted in the stochastic setting, too, but in this work we will focus only on the cotangent lift algorithm, as it possesses two advantages discussed below. The cotangent lift algorithm constructs a symplectic matrix $A$ which has the special block diagonal structure

\begin{equation}
\label{eq: Cotangent lift matrix A}
A = \left( \begin{matrix}
\Phi & 0 \\
0 & \Phi
\end{matrix} \right),
\end{equation}

\noindent
where $\Phi$ is an $n \times k$ matrix with orthogonal columns, i.e., $\Phi^T \Phi = \mathbb{I}_{k}$. Suppose snapshots of a solution are given as an $n \times 2r$ matrix $\Delta$ of the form

\begin{equation}
\label{eq: Snapshots of the solution for cotangent lift}
\Delta = [q(t_1)\,\,\ldots \,\,q(t_r)\,\,p(t_1)\,\,\ldots\,\,p(t_r)].
\end{equation}

\noindent
The SVD of $\Delta$ is truncated after the first $k$ largest singular values, similar to \eqref{eq: SVD of Delta}. The matrix $\Phi$ is then chosen as $\Phi=U_k$. With the cotangent lift matrix $A$ as in \eqref{eq: Cotangent lift matrix A}, the reduced stochastic Hamiltonian system \eqref{eq: Reduced stochastic Hamiltonian system} can be written as

\begin{align}
\label{eq: Partitioned reduced stochastic Hamiltonian system}
d_t \eta &= \phantom{-}\Phi^T\frac{\partial H}{\partial p}\Big(\Phi \eta, \Phi \chi\Big)\,dt + \sum_{\nu=1}^{m} \Phi^T\frac{\partial h_\nu}{\partial p}\Big(\Phi \eta, \Phi \chi\Big)\circ dW^\nu(t), \nonumber \\
d_t \chi &= -\Phi^T\frac{\partial H}{\partial q}\Big(\Phi \eta, \Phi \chi\Big)\,dt - \sum_{\nu=1}^{m} \Phi^T\frac{\partial h_\nu}{\partial q}\Big(\Phi \eta, \Phi \chi\Big)\circ dW^\nu(t),
\end{align}

\noindent
where $\xi = (\eta, \chi)$.

The advantage of using the cotangent lift algorithm is twofold. First, if the original system \eqref{eq: Stochastic Hamiltonian system - general} is separable, so is the reduced system \eqref{eq: Reduced stochastic Hamiltonian system}: when $H(q,p) = T(p)+V(q)$, then $\tilde H(\eta,\chi) = H(\Phi \eta, \Phi \chi) = T(\Phi \chi) + V(\Phi \eta)$; analogously for the Hamiltonians $h_\nu$. This is significant, because separable Hamiltonian systems appear often in practical applications, and many symplectic integrators, such as the stochastic St\"{o}rmer-Verlet method, become explicit in that case (see Section~\ref{sec: PSD time integration}). Second, the PSD method with the cotangent lift symplectic matrix \eqref{eq: Cotangent lift matrix A} preserves also the Lagrange-d'Alembert structure of stochastic forced Hamiltonian systems (see Section~\ref{sec: Model reduction for stochastic forced Hamiltonian systems}).
 
\subsection{Symplectic Discrete Empirical Interpolation Method}
\label{sec: Symplectic Discrete Empirical Interpolation Method}

Just like in the case of the POD method, if the drift and diffusion terms in the reduced model \eqref{eq: Reduced stochastic Hamiltonian system} are complicated nonlinear functions, the PSD method may not bring any computational savings. A technique called the Symplectic Discrete Empirical Interpolation Method (SDEIM), which applies DEIM to approximate the symplectic projection of the nonlinear terms, has been developed in \cite{PengMohseni2016}. We will argue that this technique can be adapted also in the stochastic setting. For completeness and for the benefit of the reader, below we briefly outline the main ideas of SDEIM. The further details can be found in \cite{PengMohseni2016}. Let the gradient of the Hamiltonian function $H$ be split into the linear and nonlinear parts as

\begin{equation}
\label{eq: Gradient of H - decomposition}
\nabla_u H(u)= Lu + a_N(u).
\end{equation}

\noindent
Similar to \eqref{eq: DEIM approximation of f_N}, one can use DEIM to approximate the nonlinear vector term $a_N(u)$. The SDEIM approximation of the drift term in \eqref{eq: Reduced stochastic Hamiltonian system} then takes the form

\begin{equation}
\label{eq: SDEIM approximation of the drift term}
A^+ \mathbb{J}_{2n} \nabla_u H(u) \approx \mathbb{J}_{2k}\bar L \xi + \mathbb{J}_{2k}Wg(\xi),
\end{equation}

\noindent
where

\begin{equation}
\label{eq: L bar, W, and g(xi) for SDEIM}
\bar L = A^T L A, \qquad\qquad W = A^T \Psi (P^T \Psi)^{-1}, \qquad\qquad g(\xi) = P^T a_N(A \xi).
\end{equation}

\noindent
Note that the matrices $\bar L$ and $W$ are calculated only once at the offline stage, and at the online stage the function $g(\xi)$ evaluates only $\bar k$ components of $a_N(A \xi)$. A similar SDEIM approximation can be applied to each of the diffusion terms $A^+\mathbb{J}_{2n} \nabla_u h_\nu(u)$, thus reducing the computational complexity of the reduced system \eqref{eq: Reduced stochastic Hamiltonian system}. It should, however, be noted that the PSD reduced system with the SDEIM approximation of the drift and diffusion terms is not strictly Hamiltonian anymore. Similar to \eqref{eq: Gradient of H - decomposition}, let the gradient of the Hamiltonian functions $h_\nu$ be split into the linear and nonlinear parts as

\begin{equation}
\label{eq: Gradient of h_nu - decomposition}
\nabla_u h_\nu(u)= L_\nu u + a_{N,\nu}(u)
\end{equation}

\noindent
for $\nu=1,\dots,m$. The PSD reduced system \eqref{eq: Reduced stochastic Hamiltonian system} with the SDEIM approximation takes the form

\begin{align}
\label{eq: PSD+SDEIM reduced system}
d_t \xi = \mathbb{J}_{2k} \big( \bar L \xi + A^T \bar a_N(A \xi) \big) \, dt + \sum_{\nu=1}^m \mathbb{J}_{2k} \big( \bar L_\nu \xi + A^T \bar a_{N,\nu}(A \xi) \big) \circ dW^\nu(t),
\end{align}

\noindent
where $\bar a_{N}$ and $\bar a_{N,\nu}$ denote the DEIM approximations of $a_{N}$ and $a_{N,\nu}$, respectively. In the following theorem, which is a stochastic generalization of Theorem~5.1 in \cite{PengMohseni2016}, we show in what sense the SDEIM method approximates the properties of a PSD reduced stochastic Hamiltonian system.
\newpage
\begin{thm}
\label{thm: Theorem on dE}
Let the Hamiltonians of the reduced system \eqref{eq: Reduced stochastic Hamiltonian system} satisfy

\begin{equation}
\label{eq: Poisson bracket of the reduced Hamiltonians}
\{\tilde H, \tilde h_\nu\} = (\nabla_\xi \tilde H)^T \mathbb{J}_{2k} \nabla_\xi \tilde h_\nu = 0
\end{equation}

\noindent
for $\nu=1,\dots,m$, and let $\xi(t)$ be the solution of \eqref{eq: PSD+SDEIM reduced system} with the initial condition $\xi(0)=\xi_0$. Then the stochastic differential of the Hamiltonian $\tilde H$ along $\xi(t)$, that is $E(t)=\tilde H(\xi(t))$, takes the form

\begin{equation}
\label{eq: Stochastic differential of E(t)}
d_tE(t) = \gamma(t)\,dt + \sum_{\nu=1}^m \lambda_\nu(t)\circ dW^\nu(t),
\end{equation}

\noindent
with the drift and diffusion terms given by

\begin{align}
\label{eq: Drift and diffusion terms for dE}
\gamma(t) &= \big(\nabla_\xi \tilde H(\xi)\big)^T \mathbb{J}_{2k} A^T\big( \bar a_N(A \xi) - a_N(A \xi) \big), \nonumber \\
\lambda_\nu(t) &= \big(\nabla_\xi \tilde H(\xi)\big)^T \mathbb{J}_{2k} A^T\big( \bar a_{N,\nu}(A \xi) - a_{N,\nu}(A \xi) \big),
\end{align}

\noindent
for $\nu=1,\dots,m$. Moreover, upper bounds on the drift and diffusion terms are given by

\begin{align}
\label{eq: Upper bounds on drift and diffusion terms for dE}
\big|\gamma(t)\big| &\leq  C \cdot \big\| \nabla_\xi \tilde H(\xi) \big\| \cdot \big\| (\mathbb{I}-\Psi \Psi^T) a_N(A \xi) \big\|, \nonumber \\
\big|\lambda_\nu(t)\big| &\leq C_\nu \cdot \big\| \nabla_\xi \tilde H(\xi) \big\| \cdot \big\| (\mathbb{I}-\Psi_\nu \Psi_\nu^T) a_{N,\nu}(A \xi) \big\|,
\end{align}
for $\nu=1,\dots,m$, where $C=\| (P^T\Psi)^{-1} \|$ and $C_\nu=\| (P_\nu^T\Psi_\nu)^{-1} \|$ are constants, and the matrices $P$, $\Psi$ and $P_\nu$, $\Psi_\nu$ define the DEIM approximations of the nonlinear terms $a_N$ and $a_{N,\nu}$, respectively.
\end{thm}

\begin{proof}
We calculate the stochastic differential $d_t E(t)$ as

\begin{align}
\label{eq: Calculating the stochastic differential dE in the proof}
d_t E(t) &= \big(\nabla_\xi \tilde H(\xi)\big)^T\mathbb{J}_{2k} \big( \bar L \xi + A^T \bar a_N(A \xi) \big) \, dt + \sum_{\nu=1}^m \big(\nabla_\xi \tilde H(\xi)\big)^T\mathbb{J}_{2k} \big( \bar L_\nu \xi + A^T \bar a_{N,\nu}(A \xi) \big) \circ dW^\nu(t) \nonumber \\
&=\big(\nabla_\xi \tilde H(\xi)\big)^T\mathbb{J}_{2k} \big( \bar L \xi + A^T \bar a_N(A \xi) - \nabla_\xi \tilde H(\xi)\big) \, dt \nonumber\\
&\qquad\qquad\qquad\qquad\qquad\quad + \sum_{\nu=1}^m \big(\nabla_\xi \tilde H(\xi)\big)^T\mathbb{J}_{2k} \big( \bar L_\nu \xi + A^T \bar a_{N,\nu}(A \xi) - \nabla_\xi \tilde h_\nu(\xi) \big) \circ dW^\nu(t) \nonumber \\
&=\big(\nabla_\xi \tilde H(\xi)\big)^T \mathbb{J}_{2k} A^T\big( \bar a_N(A \xi) - a_N(A \xi) \big)\,dt \nonumber \\
&\qquad\qquad\qquad\qquad\qquad\quad+\sum_{\nu=1}^m\big(\nabla_\xi \tilde H(\xi)\big)^T \mathbb{J}_{2k} A^T\big( \bar a_{N,\nu}(A \xi) - a_{N,\nu}(A \xi) \big)\circ dW^\nu(t) \nonumber \\
&=\gamma(t)\,dt + \sum_{\nu=1}^m \lambda_\nu(t)\circ dW^\nu(t),
\end{align}
where in the first equality we used the rules of Stratonovich calculus and substituted \eqref{eq: PSD+SDEIM reduced system}, and in the second equality we used $\{\tilde H, \tilde H\}=0$ and $\{\tilde H, \tilde h_\nu\}=0$. The upper bounds \eqref{eq: Upper bounds on drift and diffusion terms for dE} are obtained like in the proof of Theorem~5.1 in \cite{PengMohseni2016}.\\
\end{proof}

Even though the system \eqref{eq: PSD+SDEIM reduced system} is not necessarily Hamiltonian, Theorem~\ref{thm: Theorem on dE} shows that $|\gamma(t)|\rightarrow 0$ and $|\lambda_\nu(t)|\rightarrow 0$ when $ \big\| (\mathbb{I}-\Psi \Psi^T) a_{N}(A \xi) \big\| \rightarrow 0$ and $\big\| (\mathbb{I}-\Psi_\nu \Psi_\nu^T) a_{N,\nu}(A \xi) \big\|\rightarrow 0$.

\subsection{Proper Symplectic Decomposition for stochastic forced Hamiltonian systems}
\label{sec: Model reduction for stochastic forced Hamiltonian systems}

The PSD method described in Section~\ref{sec: Proper Symplectic Decomposition} can also be applied to Hamiltonian systems subject to external forcing. Stochastic forced Hamiltonian systems take the form 

\begin{align}
\label{eq: Stochastic dissipative Hamiltonian system}
d_t q &= \frac{\partial H}{\partial p}dt + \sum_{\nu=1}^m\frac{\partial h_\nu}{\partial p}\circ dW^\nu(t), \nonumber \\
d_t p &= \bigg[-\frac{\partial H}{\partial q} + F(q,p) \bigg] dt + \sum_{\nu=1}^m \bigg[-\frac{\partial h_\nu}{\partial q}+f_\nu(q,p)\bigg]\circ dW^\nu(t),
\end{align}

\noindent
where $H=H(q,p)$ and $h_\nu=h_\nu(q,p)$ for $\nu=1,\ldots,m$ are the Hamiltonian functions, $F=F(q,p)$ and $f_\nu=f_\nu(q,p)$ are the forcing terms, and $W(t)=(W^1(t),\ldots,W^m(t))$ is the standard $m$-dimensional Wiener process. Applications of such systems arise in many models in physics, chemistry, and biology. Particular examples include molecular dynamics (see, e.g., \cite{Beard2000}, \cite{Izaguirre2001}, \cite{Lacasta2004}, \cite{Skeel1999}), dissipative particle dynamics (see, e.g., \cite{Ripoll2001}), investigations of the dispersion of passive tracers in turbulent flows (see, e.g., \cite{Sawford2001}, \cite{Thomson1987}), energy localization in thermal equilibrium (see, e.g., \cite{Reigada1999}), lattice dynamics in strongly anharmonic crystals (see, e.g., \cite{Gornostyrev1996}), description of noise induced transport in stochastic ratchets (see, e.g., \cite{Landa1998}), and collisional kinetic plasmas (\cite{Kleiber2011}, \cite{KrausTyranowski2019}, \cite{Sonnendrucker2015}, \cite{TyranowskiVlasovMaxwell}). While their stochastic flow is not symplectic in general, stochastic forced Hamiltonian systems have an underlying variational structure, that is, their solutions satisfy the stochastic Lagrange-d'Alembert principle (see \cite{KrausTyranowski2019}). It is therefore beneficial to preserve this variational structure both in time integration and in deriving reduced models. Stochastic Lagrange-d'Alembert schemes for time integration of \eqref{eq: Stochastic dissipative Hamiltonian system}, first proposed in \cite{KrausTyranowski2019}, demonstrate better accuracy and stability properties in long-time simulations than non-geometric stochastic methods.

It was shown in \cite{PengMohseniProceedings2016} and \cite{PengMohseniPreprint2016} that the PSD method with the cotangent lift algorithm preserves the Lagrange-d'Alembert structure of deterministic forced Hamiltonian systems (see also \cite{AfkhamHesthaven2019}, \cite{HesthavenReview2021}). It is straightforward to verify that this also holds for stochastic forced Hamiltonian systems. Indeed, substituting $q=\Phi \eta$ and $p=\Phi \chi$ in \eqref{eq: Stochastic dissipative Hamiltonian system} yields, similar to \eqref{eq: Partitioned reduced stochastic Hamiltonian system}, a reduced system of the form

\begin{align}
\label{eq: Reduced stochastic forced Hamiltonian system}
d_t \eta &= \frac{\partial \tilde H}{\partial \chi}\Big(\eta, \chi\Big)\,dt + \sum_{\nu=1}^m\frac{\partial \tilde h_\nu}{\partial \chi}\Big(\eta, \chi\Big)\circ dW^\nu(t), \nonumber \\
d_t \chi &= \bigg[-\frac{\partial \tilde H}{\partial \eta}\Big(\eta, \chi\Big) + \tilde F(\eta, \chi) \bigg] dt + \sum_{\nu=1}^m \bigg[-\frac{\partial \tilde h_\nu}{\partial \eta}\Big(\eta, \chi\Big)+\tilde f_\nu(\eta, \chi)\bigg]\circ dW^\nu(t),
\end{align}

\noindent
with 

\begin{align}
\label{eq: Hamiltonians and forces for the reduced stochastic forced Hamiltonian system}
\tilde H(\eta, \chi) &= H(\Phi \eta, \Phi \chi),  & \tilde F(\eta, \chi) &= \Phi^T F(\Phi \eta, \Phi \chi), \nonumber \\
\tilde h_\nu(\eta, \chi) &= h_\nu(\Phi \eta, \Phi \chi), & \tilde f_\nu(\eta, \chi) &= \Phi^T f_\nu(\Phi \eta, \Phi \chi),
\end{align}

\noindent
that is, also a stochastic forced Hamiltonian system. The matrix $\Phi$ can be constructed from empirical data as described in Section~\ref{sec: Proper Symplectic Decomposition}.

\subsection{Time integration}
\label{sec: PSD time integration}

The full \eqref{eq: Stochastic Hamiltonian system in terms of J} and reduced \eqref{eq: Reduced stochastic Hamiltonian system} models can be solved numerically using any general purpose stochastic scheme mentioned in Section~\ref{sec: POD Time integration}. However, since these models are Hamiltonian, it is advisable to integrate them using structure-preserving methods. Stochastic symplectic integrators, similar to their deterministic counterparts, preserve the symplecticity of the Hamiltonian flow and demonstrate good energy behavior in long-time simulations (see \cite{Anmarkrud2017}, \cite{AntonWeak2014}, \cite{AntonGlobalError2013}, \cite{Anton2013}, \cite{BrehierCohen2021}, \cite{Burrage2012}, \cite{Burrage2014}, \cite{ChenCohen2020}, \cite{Cohen2014}, \cite{CohenVilmart2021}, \cite{Cordoni2022}, \cite{AntonHighOrder2014}, \cite{HolmTyranowskiSolitons}, \cite{HolmTyranowskiGalerkin}, \cite{HongSunBook2023}, \cite{Hong2015}, \cite{KrausTyranowski2019}, \cite{MaDing2012}, \cite{MaDing2015}, \cite{MilsteinRepin2001}, \cite{MilsteinRepin}, \cite{Misawa2010}, \cite{SunWang2016}, \cite{WangPHD}, \cite{Wang2014}, \cite{Wang2017}, \cite{Zhou2017} and the references therein). In this work we will focus on two stochastic symplectic Runge-Kutta methods, namely the stochastic midpoint method \eqref{eq: Stochastic midpoint method}, which is symplectic when applied to a Hamiltonian system, and the stochastic St\"{o}rmer-Verlet method (\cite{HolmTyranowskiGalerkin}, \cite{KrausTyranowski2019}, \cite{MaDing2015}). The latter method for \eqref{eq: Stochastic Hamiltonian system - general} takes the form

\begin{align}
	\label{eq:Stochastic Stormer-Verlet method}
	\Lambda&= p_i - \frac{1}{2} \frac{\partial H}{\partial q} \big(q_i, \Lambda\big)\Delta t 
	           - \sum_{\nu=1}^m\frac{1}{2} \frac{\partial h_\nu}{\partial q} \big(q_i, \Lambda\big)\Delta W^\nu, \nonumber \\
	q_{i+1}&= q_i + \frac{1}{2} \frac{\partial H}{\partial p} \big(q_i, \Lambda\big)\Delta t
	              + \frac{1}{2} \frac{\partial H}{\partial p} \big(q_{i+1}, \Lambda\big)\Delta t
	              + \sum_{\nu=1}^m\frac{1}{2} \frac{\partial h_\nu}{\partial p} \big(q_i, \Lambda\big)\Delta W^\nu
	              + \sum_{\nu=1}^m\frac{1}{2} \frac{\partial h_\nu}{\partial p} \big(q_{i+1}, \Lambda\big)\Delta W^\nu, \nonumber \\
	p_{i+1}&= \Lambda- \frac{1}{2} \frac{\partial H}{\partial q} \big(q_{i+1}, \Lambda\big)\Delta t 
	              - \sum_{\nu=1}^m\frac{1}{2} \frac{\partial h_\nu}{\partial q} \big(q_{i+1}, \Lambda\big)\Delta W^\nu,
\end{align}

\noindent
where $\Delta t$ denotes the time step, $\Delta W = (\Delta W^1, \ldots, \Delta W^m)$ are the increments of the Wiener process, and $\Lambda$ is the internal momentum stage. The stochastic St\"{o}rmer-Verlet method is strongly convergent of order 1 for systems driven by a commutative noise, and of order 1/2 in the non-commutative case. In case the stochastic Hamiltonian system \eqref{eq: Stochastic Hamiltonian system - general} is separable, the numerical scheme \eqref{eq:Stochastic Stormer-Verlet method} becomes explicit (see \cite{HolmTyranowskiGalerkin}).

Similarly, it is advisable that the full \eqref{eq: Stochastic dissipative Hamiltonian system} and reduced \eqref{eq: Reduced stochastic forced Hamiltonian system} models are solved using structure-preserving methods. Stochastic Lagrange-d'Alembert integrators (see \cite{KrausTyranowski2019}) generalize the notion of symplectic integrators by preserving the underlying variational structure of forced Hamiltonian systems. As shown in \cite{KrausTyranowski2019}, the stochastic St\"{o}rmer-Verlet method, when applied to a forced system, is a Lagrange-d'Alembert integrator, and takes the form

\begin{align}
	\label{eq:Stochastic Stormer-Verlet method for forced systems}
	\Lambda &= p_i + \frac{1}{2} \bigg[ -\frac{\partial H}{\partial q} \big(q_i, \Lambda \big) + F\big(q_i, \Lambda \big) \bigg] \Delta t 
	           + \frac{1}{2} \sum_{\nu=1}^m \bigg[ -\frac{\partial h_\nu}{\partial q} \big(q_i, \Lambda \big) + f_\nu\big(q_i, \Lambda \big) \bigg] \Delta W^\nu, \nonumber \\
	q_{i+1}&= q_i + \frac{1}{2} \frac{\partial H}{\partial p} \big(q_i, \Lambda \big)\Delta t
	              + \frac{1}{2} \frac{\partial H}{\partial p} \big(q_{i+1}, \Lambda \big)\Delta t
	              + \sum_{\nu=1}^m\frac{1}{2} \frac{\partial h_\nu}{\partial p} \big(q_i, \Lambda\big)\Delta W^\nu
	              + \sum_{\nu=1}^m\frac{1}{2} \frac{\partial h_\nu}{\partial p} \big(q_{i+1}, \Lambda\big)\Delta W^\nu, \nonumber \\
	p_{i+1}&= \Lambda + \frac{1}{2} \bigg[ -\frac{\partial H}{\partial q} \big(q_{i+1}, \Lambda \big) + F\big(q_{i+1}, \Lambda \big) \bigg] \Delta t 
	           + \frac{1}{2} \sum_{\nu=1}^m \bigg[ -\frac{\partial h_\nu}{\partial q} \big(q_{i+1}, \Lambda \big) + f_\nu\big(q_{i+1}, \Lambda \big) \bigg]\Delta W^\nu.
	\end{align}
	
	\noindent
If the Hamiltonians in \eqref{eq: Stochastic dissipative Hamiltonian system} are separable, the second equation in \eqref{eq:Stochastic Stormer-Verlet method for forced systems} becomes explicit. If in addition the forcing terms $F$ and $f_\nu$ have special forms, then further improvements in efficiency are possible. For instance, if the forcing terms depend linearly on $p$, as is often the case in practical applications, then the first equation is a linear equation for $\Lambda$, and can be solved using linear solvers. In case the forcing terms are independent of $p$ altogether, then the whole method becomes fully explicit.

\section{Model reduction for a large number of Monte Carlo runs}
\label{sec: Reduction of the number of Monte Carlo runs}

In Sections~\ref{sec: Model reduction for SDEs} and \ref{sec: Model reduction for stochastic Hamiltonian systems} we have demonstrated how model reduction can be used to efficiently solve high dimensional SDEs for one given realization of the Wiener process. In this section we will argue that model reduction also allows an efficient approach to another typical computational issue arising in numerical simulations of SDEs. In order to compute the statistical properties of the solution of \eqref{eq: Stochastic differential equation - general} or \eqref{eq: Stochastic Hamiltonian system - general}, one typically needs a very large number of Monte Carlo runs, that is, one needs to simulate the solution for a very large number of independent realizations of the Wiener process. We will show that when empirical data on the considered system is available, it can be used to construct a reduced model which requires less computational effort to yield the desired results. 

\subsection{Proper Orthogonal Decomposition}
\label{sec: Proper Orthogonal Decomposition for Monte Carlo}

Suppose we are interested in solving the stochastic differential equation

\begin{equation}
\label{eq: SDE for Monte Carlo}
d_t X = \Gamma(X)\,dt + B(X)\circ dW(t),
\end{equation}

\noindent
for an $N$-dimensional stochastic process $X(t)$, where $\Gamma:\mathbb{R}^N \longrightarrow \mathbb{R}^N$ is the drift function, $B:\mathbb{R}^N\longrightarrow \mathbb{R}^N$ is the diffusion function, and $W(t)$ is the standard one-dimensional Wiener process. For convenience and clarity we restrict ourselves to a one-dimensional noise; the generalization to a multidimensional Wiener process is straightforward. Suppose we would like to compute the statistical properties of the stochastic process $X(t)$, e.g., its mean, variance, etc. For this we need to solve \eqref{eq: SDE for Monte Carlo} for a large number $M$ of independent realizations of the Wiener process. Unlike in the situation considered in Section~\ref{sec: Proper Orthogonal Decomposition}, here we do not assume that the dimension $N$ is very high. The main computational cost comes from the high value of $M$. Therefore, in order to apply model reduction, let us turn this problem into the problem discussed in Section~\ref{sec: Model reduction for SDEs}. Let us consider $M$ stochastic processes $X_1, \ldots, X_M$ satisfying the system of stochastic differential equations

\begin{align}
\label{eq: System of SDEs for Monte Carlo}
d_t X_1 &= \Gamma(X_1)\,dt + B(X_1)\circ dW^1(t), \nonumber \\
        & \phantom{= F(X_1)\,dt+} \vdots\\
d_t X_M &= \Gamma(X_M)\,dt + B(X_M)\circ dW^M(t), \nonumber
\end{align}

\noindent
where $W^1(t), \ldots, W^M(t)$ are the components of the standard $M$-dimensional Wiener process. Note that the equations in \eqref{eq: System of SDEs for Monte Carlo} are decoupled from each other, and each equation is driven by an independent Wiener process $W^\nu(t)$. Therefore, $X_1, \ldots, X_M$ are independent identically distributed (i.i.d.) stochastic processes, each with the same probability density function as the original stochastic process $X$. In this sense \eqref{eq: System of SDEs for Monte Carlo} is equivalent to \eqref{eq: SDE for Monte Carlo}. The advantage is that instead of considering $M$ realizations of the $N$-dimensional stochastic process $X$, one can consider one realization of the $NM$-dimensional process $(X_1, \ldots, X_M)$. The value of any functional of $X$ can then be approximated using the law of large numbers, e.g., the expected value $\mathbb{E}[X] \approx (X_1+\ldots+X_M)/M$. Note that the system \eqref{eq: System of SDEs for Monte Carlo} has the form of \eqref{eq: Stochastic differential equation - general} with $u=(X_1, \ldots, X_M)$, $n=NM$, $m=M$, and the drift and diffusion functions given by

\begin{align}
\label{eq: Drift and diffusion functions for Monte Carlo}
a(u) = \begin{pmatrix}
\Gamma(X_1) \\
\vdots\\
\vdots\\
\vdots\\
\Gamma(X_M)
\end{pmatrix}, \qquad
b_1(u) = \begin{pmatrix}
B(X_1) \\
0\\
\vdots\\
\vdots\\
0
\end{pmatrix},\qquad
b_2(u) = \begin{pmatrix}
0\\
B(X_2) \\
0\\
\vdots\\
0
\end{pmatrix},\quad\ldots \quad 
b_M(u) = \begin{pmatrix}
0\\
\vdots\\
\vdots\\
0\\
B(X_M)
\end{pmatrix}.
\end{align}

\noindent
With this setting the POD method described in Section~\ref{sec: Model reduction for SDEs} can now be directly applied, and the corresponding reduced model is given by \eqref{eq: POD reduced stochastic differential equation}. Since the vectors in \eqref{eq: Drift and diffusion functions for Monte Carlo} have a sparse structure, it pays off to split the matrix $U_k$ into $M$ blocks of size $N\times k$ each, that is,

\begin{equation}
\label{eq: Block form of U_k}
U_k = \begin{pmatrix}
U^{(1)}\\
\vdots\\
U^{(M)}
\end{pmatrix}.
\end{equation}

\noindent
Then the drift and diffusion terms in \eqref{eq: POD reduced stochastic differential equation} can be evaluated as

\begin{align}
\label{eq: POD drift and diffusion terms for Monte Carlo}
U^T_ka(U_k\xi) = \sum_{\nu=1}^M \big(U^{(\nu)}\big)^T \Gamma\big(U^{(\nu)} \xi\big) \text{\quad and \quad} U^T_kb_\nu(U_k\xi) = \big(U^{(\nu)}\big)^T B\big(U^{(\nu)} \xi\big) \text{\quad for $\nu=1,\ldots,M$}.
\end{align}

\paragraph{Remark.} A related idea appears in \cite{BoyavalLelievre2010}, where a variance reduction method using the reduced basis paradigm is proposed. Variance reduction methods are a set of techniques to reduce the statistical error appearing in the Monte-Carlo estimation of the output expectation of a random variable. One of such techniques, the control variate method, involves introducing a correlated auxiliary variable (control variate) to reduce the variance of the estimator, leading to more accurate estimations of the expected value of a random variable. In \cite{BoyavalLelievre2010} model reduction is used for the efficient calculation of control variates for a certain time-independent functional of the solution of a given parameter-dependent SDE, rather than for constructing a lower-dimensional stochastic system like \eqref{eq: POD reduced stochastic differential equation}. The reduced bases are constructed in the space of functionals of the solution of the underlying SDE, rather than in the space of solutions themselves; and the data used to construct those bases are the values of the functionals for a selected set of parameters, rather than the snapshots of the trajectories of the underlying stochastic system like \eqref{eq: Empirical data for POD}. The approach outlined in this section is therefore conceptually different from the strategy employed in \cite{BoyavalLelievre2010}.

\subsection{Proper Symplectic Decomposition}
\label{sec: Proper Symplectic Decomposition for Monte Carlo}

Let us now consider the problem of solving the stochastic Hamiltonian system

\begin{align}
\label{eq: Stochastic Hamiltonian system for Monte Carlo}
d_t Q &= \phantom{-}\frac{\partial \bar H}{\partial P}(Q,P)\,dt + \frac{\partial \bar h}{\partial P}(Q,P)\circ dW(t), \nonumber \\
d_t P &= -\frac{\partial \bar H}{\partial Q}(Q,P)\,dt - \frac{\partial \bar h}{\partial Q}(Q,P)\circ dW(t),
\end{align}


\noindent
for $N$-dimensional stochastic processes $Q(t)$ and $P(t)$, where $\bar H:\mathbb{R}^N \times \mathbb{R}^N\longrightarrow \mathbb{R}$ and $\bar h:\mathbb{R}^N \times \mathbb{R}^N\longrightarrow \mathbb{R}$ are the Hamiltonian functions. Suppose we would like to solve \eqref{eq: Stochastic Hamiltonian system for Monte Carlo} for a large number $M$ of independent realizations of the Wiener process $W(t)$. In contrast to the scenario explored in Section~\ref{sec: Proper Symplectic Decomposition}, we do not assume that the dimension $2N$ of the system is very high. Rather, the main computational expense arises due to the large number $M$ of Monte Carlo runs. Therefore, in a spirit similar to Section~\ref{sec: Proper Orthogonal Decomposition for Monte Carlo}, let us consider $2M$ stochastic processes $Q_1,P_1, \ldots,Q_M,P_M$, with each pair $(Q_\nu, P_\nu)$ satisfying the stochastic differential system

\begin{align}
\label{eq: System of stochastic Hamiltonian systems for Monte Carlo}
d_t Q_\nu &= \phantom{-}\frac{\partial \bar H}{\partial P}(Q_\nu,P_\nu)\,dt + \frac{\partial \bar h}{\partial P}(Q_\nu,P_\nu)\circ dW^\nu(t), \nonumber \\
d_t P_\nu &= -\frac{\partial \bar H}{\partial Q}(Q_\nu,P_\nu)\,dt - \frac{\partial \bar h}{\partial Q}(Q_\nu,P_\nu)\circ dW^\nu(t),
\end{align}


\noindent
for $\nu=1,\ldots,M$, where $W^1(t), \ldots, W^M(t)$ are the components of the standard $M$-dimensional Wiener process. Note that the systems \eqref{eq: System of stochastic Hamiltonian systems for Monte Carlo} are decoupled from each other for different values of $\nu$, and each system is driven by an independent Wiener process $W^\nu(t)$. Therefore, the pairs $(Q_\nu, P_\nu)$ for $\nu=1,\ldots,M$ are independent identically distributed (i.i.d.) stochastic processes, each with the same probability density function as the original stochastic process $(Q, P)$. In that sense the system \eqref{eq: System of stochastic Hamiltonian systems for Monte Carlo} is equivalent to \eqref{eq: Stochastic Hamiltonian system for Monte Carlo}. Note that the system \eqref{eq: System of stochastic Hamiltonian systems for Monte Carlo} has the form of \eqref{eq: Stochastic Hamiltonian system - general} with $q=(Q_1,\ldots,Q_M)$, $p=(P_1,\ldots,P_M)$, $n=NM$, $m=M$, and the Hamiltonian functions given by

\begin{equation}
\label{eq: Hamiltonians for Monte Carlo}
H(q,p) = \sum_{\nu=1}^M \bar H(Q_\nu,P_\nu) \text{\qquad and \qquad} h_\nu(q,p) = \bar h(Q_\nu,P_\nu) \text{\qquad for $\nu=1,\ldots,M$}. 
\end{equation}

\noindent
With this setting the PSD method described in Section~\ref{sec: Model reduction for stochastic Hamiltonian systems} can now be applied, and the corresponding reduced stochastic Hamiltonian system is given by \eqref{eq: Partitioned reduced stochastic Hamiltonian system} for the cotangent lift algorithm. Similar to \eqref{eq: Drift and diffusion functions for Monte Carlo}, the diffusion terms in \eqref{eq: System of stochastic Hamiltonian systems for Monte Carlo} have a sparse structure, therefore it pays off to split the matrix~$\Phi$ into $M$ blocks of size $N \times k$

\begin{equation}
\label{eq: Block form of Phi}
\Phi = \begin{pmatrix}
\Phi^{(1)}\\
\vdots\\
\Phi^{(M)}
\end{pmatrix}.
\end{equation}

\noindent
Then the drift and the diffusion terms of the reduced model \eqref{eq: Partitioned reduced stochastic Hamiltonian system} can be expressed as

\begin{align}
\label{eq: PSD drift terms for Monte Carlo}
\Phi^T\frac{\partial H}{\partial p}\Big(\Phi \eta, \Phi \chi\Big) &= \sum_{\nu=1}^M \big(\Phi^{(\nu)}\big)^T\frac{\partial \bar H}{\partial P}\Big(\Phi^{(\nu)} \eta, \Phi^{(\nu)} \chi\Big), \nonumber \\
-\Phi^T\frac{\partial H}{\partial q}\Big(\Phi \eta, \Phi \chi\Big) &= -\sum_{\nu=1}^M \big(\Phi^{(\nu)}\big)^T\frac{\partial \bar H}{\partial Q}\Big(\Phi^{(\nu)} \eta, \Phi^{(\nu)} \chi\Big),
\end{align}

\noindent
and, for $\nu=1,\ldots,M$, 

\begin{align}
\label{eq: PSD diffusion terms for Monte Carlo}
\Phi^T\frac{\partial h_\nu}{\partial p}\Big(\Phi \eta, \Phi \chi\Big) &=  \big(\Phi^{(\nu)}\big)^T\frac{\partial \bar h}{\partial P}\Big(\Phi^{(\nu)} \eta, \Phi^{(\nu)} \chi\Big), \nonumber \\
-\Phi^T\frac{\partial h_\nu}{\partial q}\Big(\Phi \eta, \Phi \chi\Big) &=  -\big(\Phi^{(\nu)}\big)^T\frac{\partial \bar h}{\partial Q}\Big(\Phi^{(\nu)} \eta, \Phi^{(\nu)} \chi\Big).
\end{align}


\paragraph{Remark.} An analogous strategy can be used for the stochastic forced system \eqref{eq: Stochastic dissipative Hamiltonian system}, where similar steps as above have to be applied also to the forcing terms. For brevity, we omit presenting detailed formulas.

\section{Numerical experiments}
\label{sec: Numerical experiments}

We present the results of three numerical experiments that we have carried out to validate the methods proposed in Sections~\ref{sec: Model reduction for SDEs}, \ref{sec: Model reduction for stochastic Hamiltonian systems}, and \ref{sec: Reduction of the number of Monte Carlo runs}. In the first experiment we have applied model reduction to the semi-discretization of a stochastic Nonlinear Schr\"{o}dinger Equation, whereas in the second and third experiments model reduction has been used to reduce the computational cost of the simulations of the Kubo oscillator, unforced and forced, respectively. All computations have been performed in the Julia programming language with the help of the \emph{GeometricIntegrators.jl} library (see \cite{KrausGeometricIntegrators}).

\subsection{Stochastic Nonlinear Schr\"{o}dinger Equation}
\label{sec: Stochastic Nonlinear Schrodinger Equation}

The Nonlinear Schr\"{o}dinger Equation (NLS) is a well-known nonlinear partial differential equation (PDE) with a broad spectrum of applications, ranging from wave propagation in nonlinear media to nonlinear optics, molecular biology, quantum physics, quantum chemistry, and plasma physics (see \cite{SulemBook2007}, \cite{ZakharovManakov1974} and the references therein). Model reduction for semi-discretizations of NLS is considered in, e.g., \cite{AfkhamHesthaven2017}, \cite{Karasozen2018}. Various stochastic perturbations of NLS have been proposed in order to, e.g., take into account inhomogeneities of the media or noisy sources (see \cite{AntonCohen2018}, \cite{Bang1994}, \cite{Bang1995}, \cite{CohenDujardin2017}, \cite{BouardDebussche2001}, \cite{DebusscheMenza2002b}, \cite{DebusscheMenza2002a}, \cite{Elgin1993}, \cite{Falkovich2001}, \cite{Rasmussen1995}). The stochastic NLS equation of the form

\begin{equation}
\label{eq: Stochastic NLS}
i d_t\psi + \bigg(\frac{\partial^2 \psi}{\partial x^2}+\epsilon |\psi|^2\psi \bigg)\,dt + \beta\psi\circ dW(t) = 0,
\end{equation}

\noindent
for a complex-valued function $\psi = \psi(x,t)$, where $\epsilon$, $\beta$ are real parameters, and $i$ denotes the imaginary unit, has been proposed in \cite{Elgin1993} as a model for the propagation of optical pulses down a nonideal anomalously dispersive optical fiber, with the multiplicative noise term describing the effects of local density fluctuations in the fiber material. A similar equation has also been considered as a model of energy transfer in a monolayer molecular aggregate in the presence of thermal fluctuations (see \cite{Bang1994}, \cite{Bang1995}, \cite{Rasmussen1995}). Equation~\eqref{eq: Stochastic NLS} reduces to the deterministic NLS equation for $\beta=0$. It can be verified by a straightforward calculation that if $\psi_D(x,t)$ is a solution of the deterministic NLS equation, then $\psi(x,t) = \exp(i\beta W(t)) \psi_D(x,t)$ is a solution of \eqref{eq: Stochastic NLS} (see \cite{Elgin1993}). This in particular means that Equation~\eqref{eq: Stochastic NLS}, similar to its deterministic counterpart, also possesses solitonic solutions. By considering the real and imaginary parts of $\psi$, Equation~\eqref{eq: Stochastic NLS} can be rewritten as a system of coupled stochastic PDEs,

\begin{align}
\label{eq: Stochastic NLS in terms of q and p}
d_tq = -\bigg[ \frac{\partial^2 p}{\partial x^2} + \epsilon (q^2+p^2)p \bigg]\,dt - \beta p \circ dW(t), \qquad\quad d_tp = \bigg[ \frac{\partial^2 q}{\partial x^2} + \epsilon (q^2+p^2)q \bigg]\,dt + \beta q \circ dW(t),
\end{align}

\noindent
for real-valued functions $q=q(x,t)$ and $p=p(x,t)$, where $\psi=q+ip$. Equation~\eqref{eq: Stochastic NLS in terms of q and p} has the form of a stochastic Hamiltonian PDE, that is,

\begin{align}
\label{eq: Stochastic Hamiltonian PDE}
d_tq = \frac{\delta \mathcal{H}_0}{\delta p} \,dt + \frac{\delta \mathcal{H}_1}{\delta p}\circ dW(t), \qquad\quad d_tp = -\frac{\delta \mathcal{H}_0}{\delta q} \,dt - \frac{\delta \mathcal{H}_1}{\delta q}\circ dW(t),
\end{align}

\noindent
with the Hamiltonian functionals given by

\begin{align}
\label{eq: Hamiltonian functionals for the stochastic NLS}
\mathcal{H}_0[q,p] = \int \bigg[ \frac{1}{2}\bigg(\frac{\partial q}{\partial x} \bigg)^2 + \frac{1}{2}\bigg(\frac{\partial p}{\partial x} \bigg)^2 -\frac{\epsilon}{4}(q^2+p^2)^2\bigg]\,dx, \qquad\quad \mathcal{H}_1[q,p] = -\frac{\beta}{2}\int (q^2+p^2)\,dx.
\end{align}

\subsubsection{Semi-discretization}
\label{sec: Semi-discretization}

Suppose we would like to solve \eqref{eq: Stochastic NLS in terms of q and p} on the bounded spatial domain $[0,X_{max}]$ with periodic boundary conditions. Let us introduce a uniform spatial mesh consisting of the $N$ points $x_j=(j-1)\Delta x$ for $j=1,\ldots,N$, where $\Delta x = X_{max}/N$ is the mesh size, and let us denote $q^j(t)=q(x_j,t)$ and $p^j(t)=p(x_j,t)$. Using central differences to approximate the second derivatives in \eqref{eq: Stochastic NLS in terms of q and p}, we obtain the system of $2N$ stochastic differential equations

\begin{align}
\label{eq: Semi-discretization of the stochastic NLS}
d_tq^j &= -\bigg[ \frac{p^{j+1}-2p^j+p^{j-1}}{\Delta x^2} + \epsilon \Big((q^j)^2+(p^j)^2\Big)p^j \bigg]\,dt - \beta p^j \circ dW(t), \nonumber \\
d_tp^j &= \phantom{-}\bigg[ \frac{q^{j+1}-2q^j+q^{j-1}}{\Delta x^2} + \epsilon \Big((q^j)^2+(p^j)^2\Big)q^j \bigg]\,dt + \beta q^j \circ dW(t),
\end{align}

\noindent
for $j=1,\ldots,N$, with $q^0\equiv q^N$, $q^{N+1}\equiv q^1$, $p^0\equiv p^N$, and $p^{N+1}\equiv p^1$. Equation~\eqref{eq: Semi-discretization of the stochastic NLS} is a stochastic Hamiltonian system \eqref{eq: Stochastic Hamiltonian system - general} with the Hamiltonians

\begin{align}
\label{eq: Hamiltonians for the semi-discretization of NLS}
H = \sum_{j=1}^N\bigg[ \frac{1}{2}\bigg(\frac{q^{j+1}-q^j}{\Delta x} \bigg)^2 + \frac{1}{2}\bigg(\frac{p^{j+1}-p^j}{\Delta x} \bigg)^2 -\frac{\epsilon}{4}\Big((q^j)^2+(p^j)^2\Big)^2\bigg], \quad\quad\,\, h = -\frac{\beta}{2} \sum_{j=1}^N \Big((q^j)^2+(p^j)^2\Big),
\end{align}

\noindent
which are related to the discretizations of the Hamiltonian functionals \eqref{eq: Hamiltonian functionals for the stochastic NLS}. Since the noise is one-dimensional, for simplicity we write $h\equiv h_1$.  One can check that the condition \eqref{eq: Poisson bracket of the Hamiltonians} is satisfied, therefore the Hamiltonian $H$ is almost surely preserved on solutions to \eqref{eq: Semi-discretization of the stochastic NLS}. The system \eqref{eq: Semi-discretization of the stochastic NLS} can be integrated in time using general purpose stochastic methods (see Section~\ref{sec: POD Time integration}) or stochastic symplectic schemes (see Section~\ref{sec: PSD time integration}).

\subsubsection{Empirical data}
\label{sec: Empirical data}

Suppose we have the following computational problem: we would like to scan the domains of $\beta$ and~$\epsilon$, that is, compute the numerical solution of \eqref{eq: Semi-discretization of the stochastic NLS} for a single realization of the Wiener process for a large number of values of $\beta$ and $\epsilon$. Given that in practical applications the system \eqref{eq: Semi-discretization of the stochastic NLS} is high-dimensional, this task is computationally intensive. Model reduction can alleviate this substantial computational cost. One can carry out full-scale computations only for a selected number of values of $\beta$ and $\epsilon$. These data can then be used to identify reduced models, as described in Sections~\ref{sec: Model reduction for SDEs} and~\ref{sec: Model reduction for stochastic Hamiltonian systems}. The lower-dimensional equations \eqref{eq: POD reduced stochastic differential equation} or \eqref{eq: Partitioned reduced stochastic Hamiltonian system} can then be solved more efficiently for other values of $\beta$ and $\epsilon$, thus reducing the overall computational cost. Let us consider the initial conditions

\begin{align}
\label{eq: Initial conditions for NLS}
q^j(0) = \sqrt{2} \sech (x_j-x_c) \cos \frac{c}{2}(x_j-x_c), \qquad p^j(0) = \sqrt{2} \sech (x_j-x_c) \sin \frac{c}{2}(x_j-x_c),
\end{align}

\noindent
for $j=1,\ldots,N$. These initial conditions correspond to a soliton for \eqref{eq: Stochastic NLS} centered at $x=x_c$ and propagating with the speed $c$ in the case when $\epsilon=1$ (see \cite{SulemBook2007}, \cite{ZakharovManakov1974}). For our experiment, we calculated the numerical solution to the full model \eqref{eq: Semi-discretization of the stochastic NLS} for 24 pairs of (arbitrarily selected) values of the parameters $(\beta, \epsilon)$, with $\beta$ and $\epsilon$ taking the following values:

\begin{equation}
\label{eq: beta values}
\beta = \; 0.12, \; 0.14, \; 0.16, \; 0.18, \qquad\qquad \epsilon = \; 0.95, \; 0.97, \; 0.99, \; 1.01, \; 1.03, \; 1.05.
\end{equation}

\noindent
Computations were carried out for $N=256$ mesh points using the stochastic midpoint method \eqref{eq: Stochastic midpoint method} over the time interval $0\leq t \leq 200$ with the time step $\Delta t=0.01$. The remaining parameters were $X_{max}=60$, $\Delta x \approx 0.2344$, $c=1$, and $x_c=30$. The same sample path of the Wiener process was used for all simulations (i.e., the same seed for the random number generator was used when calculating a sample path of $W(t)$). The generated data for all values of $\beta$ and $\epsilon$ were put together and used to form the snapshot matrices \eqref{eq: Empirical data for POD} and \eqref{eq: Snapshots of the solution for cotangent lift}. For instance, for the POD snapshot matrix \eqref{eq: Empirical data for POD} we used

\begin{equation}
\label{eq: Snapshots of the solution in the experiment}
\Delta = [u(t_1;\beta_1, \epsilon_1) \,\, u(t_2;\beta_1, \epsilon_1) \,\, u(t_3;\beta_1, \epsilon_1) \,\, \ldots \,\, u(t_1;\beta_1, \epsilon_2) \,\, u(t_2;\beta_1, \epsilon_2)\,\, u(t_3;\beta_1, \epsilon_2) \,\, \ldots].
\end{equation}

\noindent
Then, following the description of the methods in Sections~\ref{sec: Model reduction for SDEs} and~\ref{sec: Model reduction for stochastic Hamiltonian systems}, reduced models \eqref{eq: POD reduced stochastic differential equation} and \eqref{eq: Reduced stochastic Hamiltonian system}  were derived. Note that the drift term in \eqref{eq: Semi-discretization of the stochastic NLS} has a nonlinear part, therefore its DEIM approximation \eqref{eq: DEIM approximation of f_N} was also constructed. The decay of the singular values for the POD and PSD reductions, and for the DEIM approximation of the nonlinear term is depicted in Figure~\ref{fig: Singular values for NLS}.   In general, a sufficiently fast decay of the singular values indicates that it is possible to obtain an accurate low-rank approximation of the snapshot matrix \eqref{eq: SVD of Delta} for a small value of $k$, which is a necessary condition for the success of reduced model simulations. Some problems are more amenable to model reduction than others (see, e.g., \cite{BuffaMaday2012}, \cite{Ohlberger2016}, \cite{PinkusBook1985}). The decay in Figure~\ref{fig: Singular values for NLS} is not particularly fast (compare with Figure~\ref{fig: Singular values for Kubo} and Figure~\ref{fig: Singular values for forced Kubo}), which means that perhaps one should not expect a very significant reduction of the dimension. This is typical behavior for wave-like phenomena and transport problems (see \cite{GreifUrban2019}, \cite{Ohlberger2016}), for which more advanced model reduction techniques have been developed, such as online adaptive methods (\cite{CarlbergAdaptive2015}, \cite{Peherstorfer2015}), shifted PODs (\cite{Reiss2018}), or nonlinear manifold reduction methods (\cite{RimPeherstorfer2023}). Nevertheless, the main purpose of our experiment is to compare the performance of POD and PSD reductions for stochastic systems, even if the level of reduction is not high. In addition, as described below in Section~\ref{sec: Reduced model simulations for NLS}, in this particular case it is still possible to obtain computational advantage over full-model simulations. For other applications of POD and PSD to deterministic systems possessing wave-like solutions see, e.g., \cite{AfkhamHesthaven2017}, \cite{AfkhamHesthaven2018}, \cite{AfkhamHesthaven2019}, \cite{Karasozen2018}, \cite{PengMohseni2016}.  

\begin{figure}[tbp]
	\centering
		\includegraphics[width=.8\textwidth]{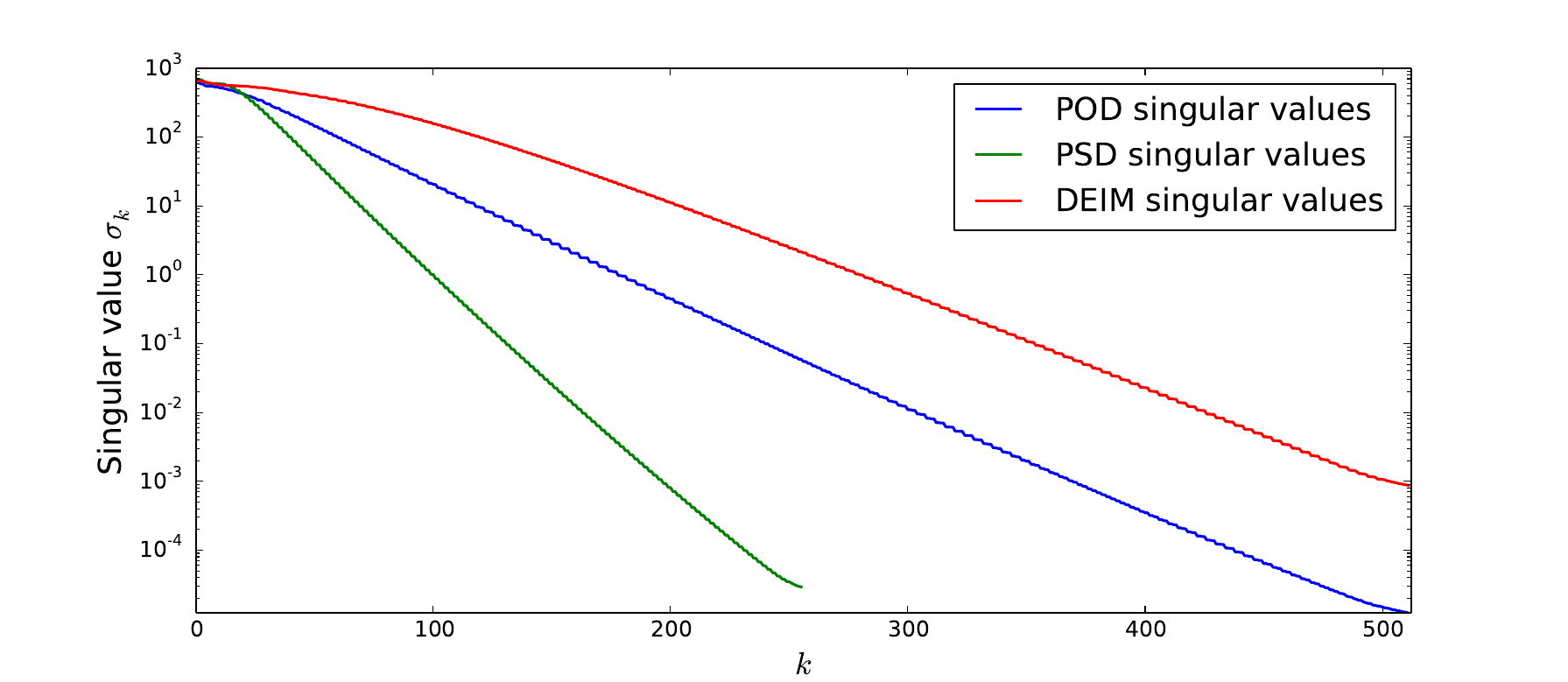}
		\caption{ The decay of the singular values for the POD and PSD reductions, and for the DEIM approximation of the nonlinear term for the empirical data ensemble for the stochastic NLS equation. }
		\label{fig: Singular values for NLS}
\end{figure}

\subsubsection{Reduced model simulations}
\label{sec: Reduced model simulations for NLS}

\begin{figure}[tbp]
	\centering
		\includegraphics[width=\textwidth]{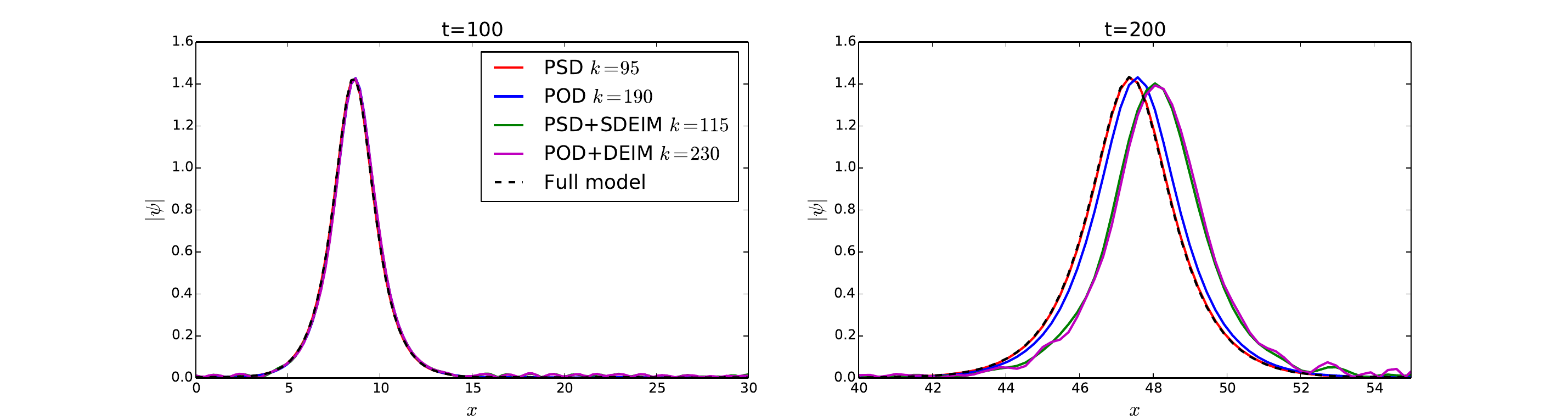}
		\caption{ The solution $|\psi(x,t)|=|q(x,t)+ip(x,t)|$ of the stochastic NLS equation with the parameters $\beta=0.15$ and $\epsilon=1$ at times $t=100$ (\emph{Left}) and $t=200$ (\emph{Right}) obtained with the help of the reduced models integrated with the stochastic midpoint method using the time step $\Delta t=0.01$. While at $t=100$ all simulations resolve the soliton relatively well, at $t=200$ the PSD models yield a more accurate solution than the POD models of the same dimension. The POD+DEIM and PSD+SDEIM simulations capture the propagation of the soliton, but create some spurious oscillations in its tail.}
		\label{fig: NLS Solution---Midpoint}
\end{figure}

\begin{figure}[tbp]
	\centering
		\includegraphics[width=\textwidth]{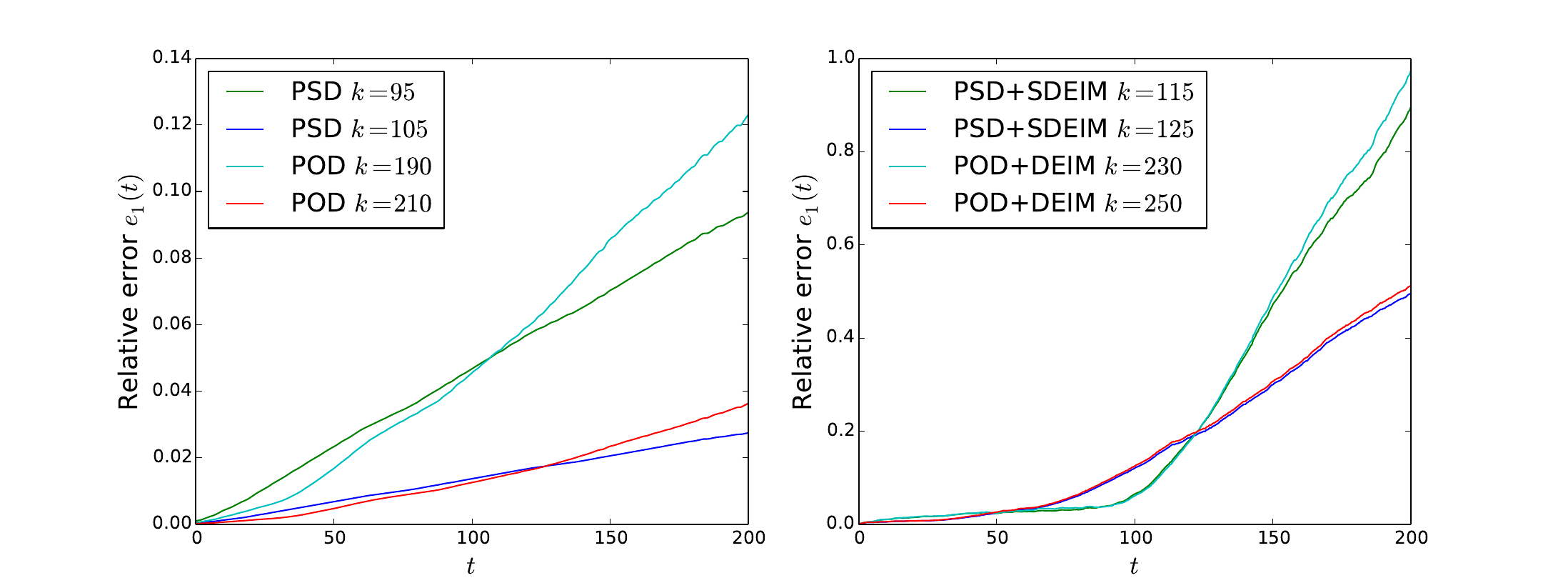}
		\caption{ The relative error $e_1(t)$ as a function of time for several example reduced model simulations of the stochastic NLS equation using the stochastic midpoint method with the time step $\Delta t=0.01$.}
		\label{fig: NLS Instant Error---Midpoint}
\end{figure}

\begin{figure}[tbp]
	\centering
		\includegraphics[width=\textwidth]{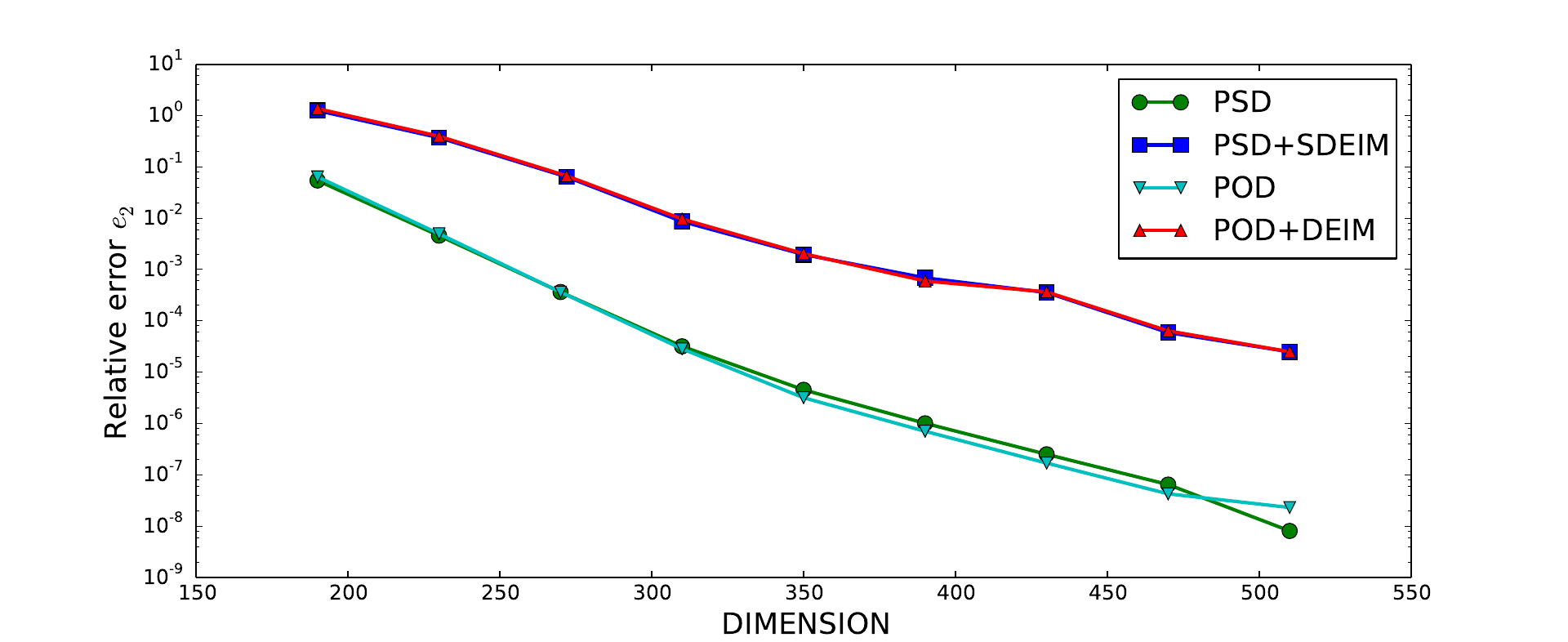}
		\caption{ The relative error $e_2$ for the reduced model simulations of the stochastic NLS equation using the stochastic midpoint method with the time step $\Delta t=0.01$ is depicted as a function of the dimension of the reduced system. Recall that the dimension of the reduced model is equal to $k$ for the POD methods, and $2k$ for the PSD methods.}
		\label{fig: NLS Global Error---Midpoint}
\end{figure}

\begin{figure}[tbp]
	\centering
		\includegraphics[width=\textwidth]{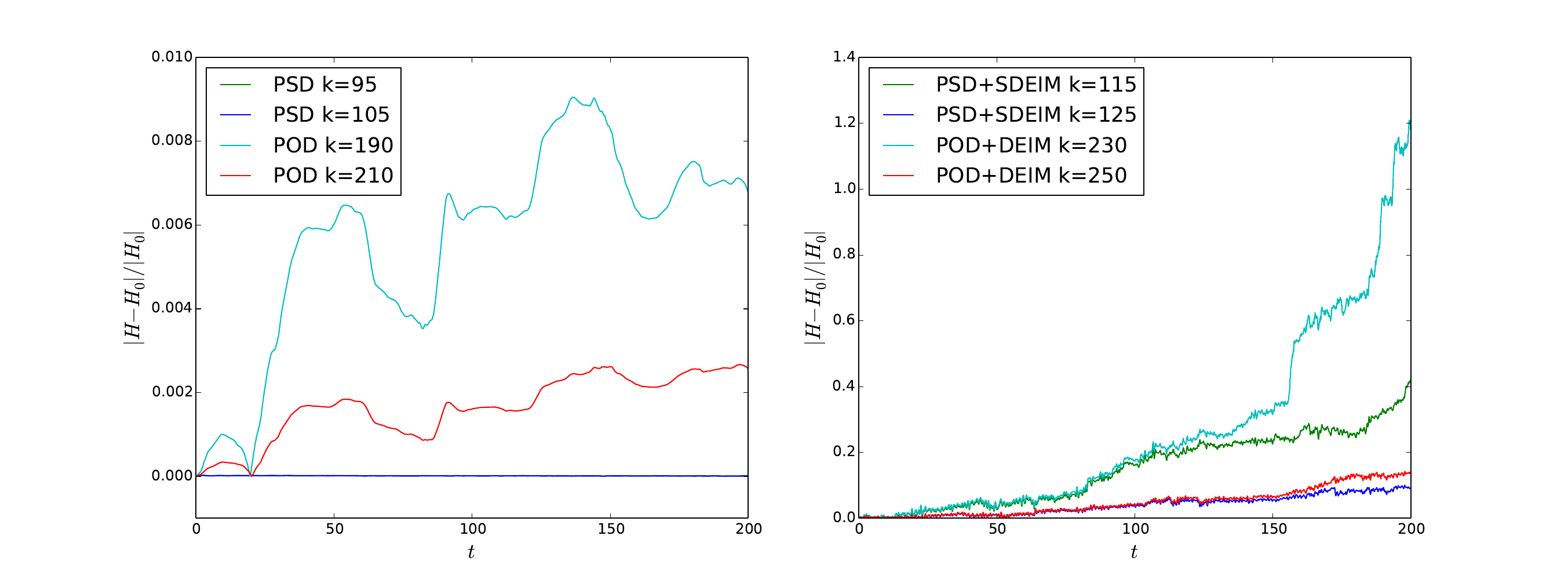}
		\caption{ The relative error of the Hamiltonian as a function of time for several example reduced model simulations of the stochastic NLS equation using the stochastic midpoint method with the time step $\Delta t=0.01$ is depicted, where $H_0$ denotes the initial value of the Hamiltonian. The PSD simulations preserve the Hamiltonian nearly exactly, in contrast to the POD simulations. Although the PSD reduced models in combination with the SDEIM approximation do not show such a good behavior, they still preserve the Hamiltonian better than the POD models of the same dimension in combination with the DEIM approximation. Note that the plots for the PSD method with $k=95$ and $k=105$ overlap very closely and are therefore indistinguishable.  }
		\label{fig: NLS Energy---Midpoint}
\end{figure}

\begin{figure}[tbp]
	\centering
		\includegraphics[width=\textwidth]{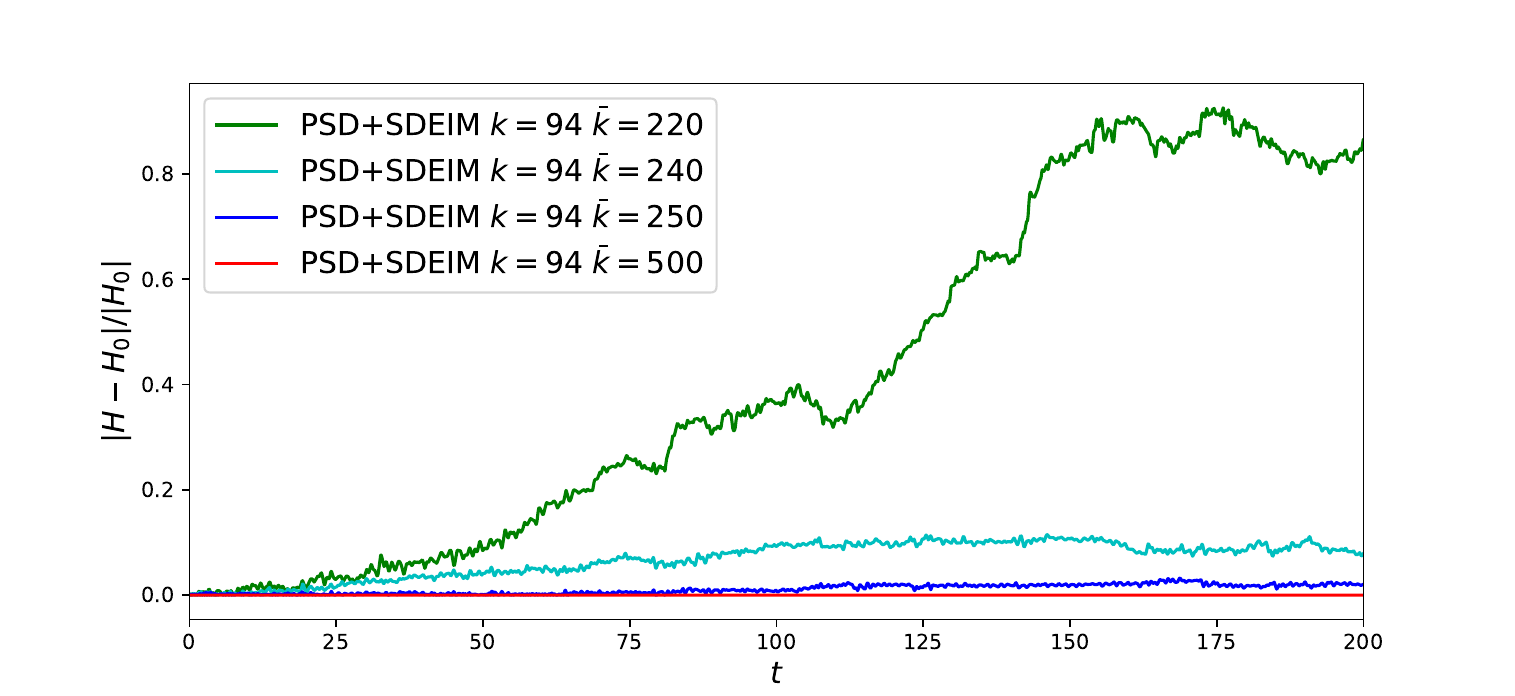}
		\caption{ The relative error of the Hamiltonian as a function of time for several example PSD+SDEIM reduced model simulations of the stochastic NLS equation using the stochastic midpoint method with the time step $\Delta t=0.01$ is depicted, where $H_0$ denotes the initial value of the Hamiltonian. A fixed dimension $2k=188$ of the reduced system is used, but the number $\bar k$ of the calculated components of the nonlinear term in the SDEIM approximation is varied. We can see that the larger the number $\bar k$ is, the less the numerical evolution of the Hamiltonian deviates from the initial value, in line with the prediction based on Theorem~\ref{thm: Theorem on dE}.}
		\label{fig: NLS Energy---SDEIM convergence}
\end{figure}

The reduced models obtained in Section~\ref{sec: Empirical data} from the empirical data can now be solved for any other desired values of the parameters $\beta$ and $\epsilon$.  To test the accuracy of the considered model reduction methods, we have compared the results of the reduced model simulations to a full-scale reference solution for a single arbitrary choice of the values of the parameters.  The reference solution for $\beta=0.15$ and $\epsilon=1$ was calculated on the time interval $0\leq t \leq 200$ in the same way as the empirical data in Section~\ref{sec: Empirical data}. Note that for this choice of $\epsilon$ the reference solution is a soliton. The reduced models were solved numerically on the same time interval using the stochastic midpoint, $R2$, and Heun methods. Note that when applied to a Hamiltonian system, the stochastic midpoint method is a symplectic integrator, while the $R2$ and Heun methods are not. The approximations of the nonlinear term were carried out with $\bar k = k$ for the POD+DEIM simulations (i.e., the simulations of the POD reduced models combined with the DEIM approximation),  and with $\bar k = 2k$ for the PSD+SDEIM simulations (i.e., the simulations of the PSD reduced models combined with the SDEIM approximation),  where $k$ is the number of the modes kept in the SVD decompositions of \eqref{eq: Empirical data for POD} and \eqref{eq: Snapshots of the solution for cotangent lift}, respectively. This way the number of the calculated components of the nonlinear term was commensurate with the dimension of the reduced system. The choice of $k$ is a compromise between the speed and the accuracy: the smaller $k$ the faster the computation, but also the larger the projection error. In practice, one may choose $k$ based on the initial value of the error \eqref{eq: Relative error at time t for NLS}, i.e., the value of the projection error for the initial condition. For instance, in our experiment, the initial relative error for the POD simulations was equal to $5\cdot 10^{-4}$ for $k=190$, and $2.2\cdot 10^{-4}$ for $k=210$ (see also Figure~\ref{fig: NLS Instant Error---Midpoint}). 

The simulations using the stochastic midpoint method were carried out with the time step $\Delta t = 0.01$. The numerical solution for $|\psi(x,t)|$ obtained from the reduced models at times $t=100$ and $t=200$ is compared against the reference solution in Figure~\ref{fig: NLS Solution---Midpoint}. We can see that while at $t=100$ all simulations resolve the soliton relatively well, after a longer time the PSD models yield a more accurate solution than the POD models of the same dimension. We also see that both the POD+DEIM and PSD+SDEIM simulations capture the propagation of the soliton, but create some spurious oscillations in its tail. As a measure of accuracy of the reduced systems at time $t$ we take the relative $L^2([0,60])$ error, that is,

\begin{equation}
\label{eq: Relative error at time t for NLS}
e_1(t)=\frac{\| \psi(\cdot,t)-\psi_{\text{ref}}(\cdot,t)\|_1}{\| \psi_{\text{ref}}(\cdot,t)\|_1},
\end{equation}

\noindent
where $\psi_{\text{ref}}$ is the reference full-model solution, as described above, and $\| \psi(\cdot,t)\|_1 = \sqrt{\int_0^{60} |\psi(x,t)|^2\,dx}$. The relative error $e_1(t)$ for several example reduced model simulations is depicted in Figure~\ref{fig: NLS Instant Error---Midpoint}. The accuracy of the reduced model simulations is improved, as the number of modes $k$ is increased. The convergence of the solutions of the reduced systems to the full-model solution is depicted in Figure~\ref{fig: NLS Global Error---Midpoint}, where the relative error $e_2$ is defined as

\begin{equation}
\label{eq: Global relative error for NLS}
e_2=\frac{\| \psi-\psi_{\text{ref}}\|_2}{\| \psi_{\text{ref}}\|_2},
\end{equation}

\noindent 
and $\| \psi\|_2 = \sqrt{\int_0^{200}\int_0^{60} |\psi(x,t)|^2\,dxdt}$ is the $L^2([0,60]\times[0,200])$ norm. A clear advantage of using the PSD method in combination with symplectic integration in time is the preservation of the Hamiltonian $H$ in long-time simulations. The relative error of the Hamiltonian for several example reduced model simulations is depicted in Figure~\ref{fig: NLS Energy---Midpoint}. We can see that, in contrast to the POD simulations, the PSD simulations nearly exactly preserve the Hamiltonian. As pointed out in Section~\ref{sec: Symplectic Discrete Empirical Interpolation Method}, the PSD method in combination with the SDEIM approximation does not result in a Hamiltonian system. Nevertheless, after a long integration time the PSD+SDEIM reduced models preserve the Hamiltonian better than the POD+DEIM reduced models of the same dimension. In order to further investigate the behavior of the SDEIM approximation in the stochastic setting, and to provide a numerical validation of Theorem~\ref{thm: Theorem on dE}, additional simulations for a fixed $k=94$ and varying $\bar k$ were carried out. The results presented in Figure~\ref{fig: NLS Energy---SDEIM convergence} indicate that the larger the number $\bar k$ is, the better the reduced system \eqref{eq: PSD+SDEIM reduced system} preserves the Hamiltonian of the system \eqref{eq: Reduced stochastic Hamiltonian system}.   It is also worth noting that despite the slow decay of the singular values in Figure~\ref{fig: Singular values for NLS} and the relatively high dimension of the reduced system needed to obtain satisfactory accuracy (see Figure~\ref{fig: NLS Global Error---Midpoint}), it is still possible to achieve reasonable computational speedup compared to the full model simulations. The average runtimes of the POD and PSD reduced simulations using the stochastic midpoint method for several dimensions of the reduced system are summarized in Table~\ref{tab:Runtimes for the POD and PSD simulations for NLS}.  

\begin{table*}
	\centering
		\begin{tabular}{c|c|c}
		  \textbf{dim} & \textbf{POD} & \textbf{PSD}\\
			\hline
			190 & 7m:32.4s \emph{(speedup 50.5\%)}  & 2m:50s \emph{(speedup 81.4\%)} \\
			210 & 9m:22.3s \emph{(speedup 38.4\%)}  & 3m:29.3s \emph{(speedup 77.1\%)} \\
			230 & 11m:23.8s \emph{(speedup 25.1\%)} & 4m:16.4s \emph{(speedup 71.9\%)} \\
			250 & 13m:05.3s \emph{(speedup 14\%)} & 5m:00.5s \emph{(speedup 67.1\%)} \\
		\end{tabular}
	\caption{Runtimes of the POD and PSD reduced simulations of the stochastic NLS equation using the midpoint method (averaged over 5 runs) for different values of the dimension of the reduced system (the dimension is equal to $k$ for POD reduced models, and $2k$ for PSD reduced models). The simulation of the full model (of dimension $2N=512$) with the midpoint method took 15m:13.2s (averaged over 5 runs).  }
	\label{tab:Runtimes for the POD and PSD simulations for NLS}
\end{table*}

\begin{figure}[tbp]
	\centering
		\includegraphics[width=\textwidth]{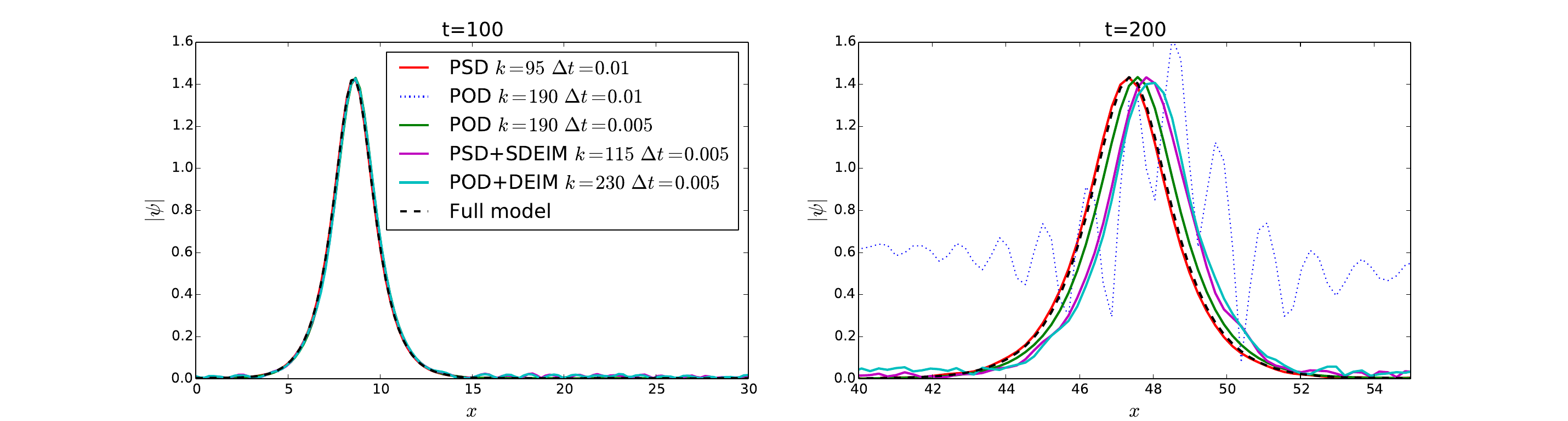}
		\caption{ The solution $|\psi(x,t)|=|q(x,t)+ip(x,t)|$ of the stochastic NLS equation with the parameters $\beta=0.15$ and $\epsilon=1$ at times $t=100$ (\emph{Left}) and $t=200$ (\emph{Right}) obtained with the help of the reduced models integrated with the stochastic $R2$ method. While at $t=100$ all simulations resolve the soliton relatively well, at $t=200$ the PSD models yield a more accurate solution than the POD models of the same dimension and integrated with the same time step. The POD simulation with $k=190$ and $\Delta t=0.01$ eventually becomes unstable. The POD+DEIM and PSD+SDEIM simulations capture the propagation of the soliton, but create some spurious oscillations in its tail.}
		\label{fig: NLS Solution---R2}
\end{figure}

\begin{figure}[tbp]
	\centering
		\includegraphics[width=\textwidth]{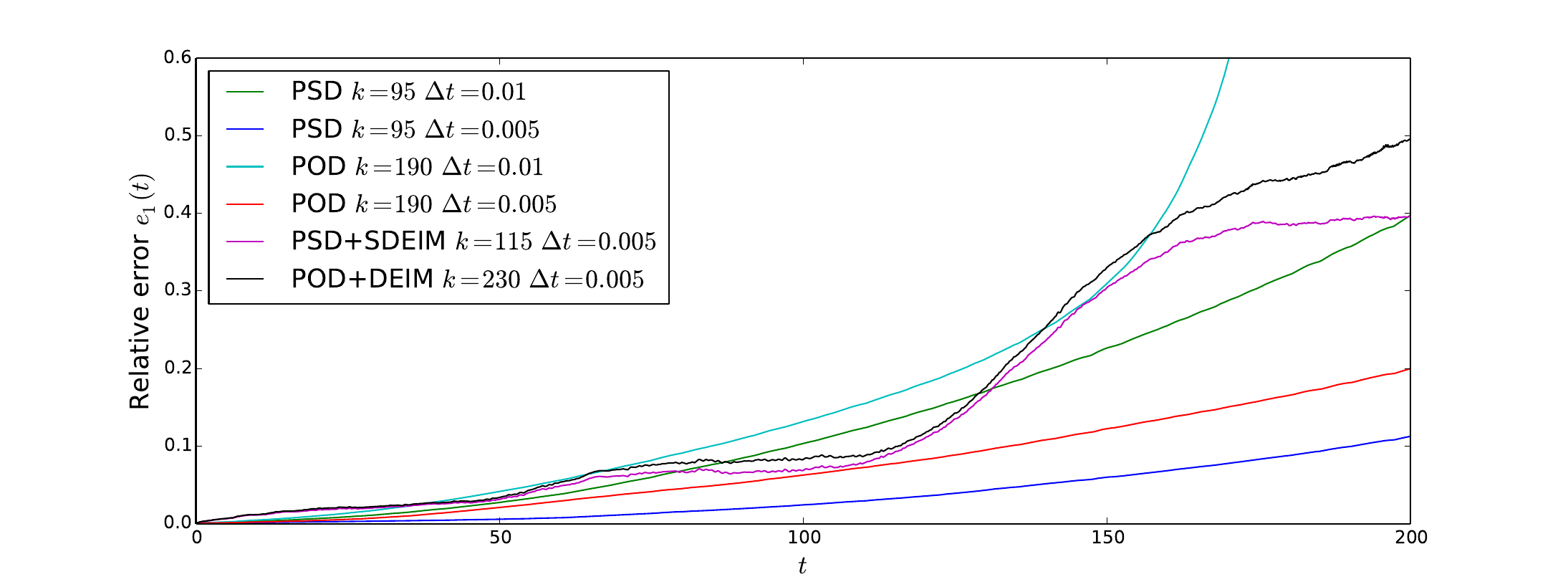}
		\caption{ The relative error $e_1(t)$ as a function of time for several example reduced model simulations of the stochastic NLS equation using the stochastic $R2$ method. Note that the error for the POD method with $k=190$ and $\Delta t = 0.01$ blows up.}
		\label{fig: NLS Instant Error---R2}
\end{figure}

\begin{figure}[tbp]
	\centering
		\includegraphics[width=\textwidth]{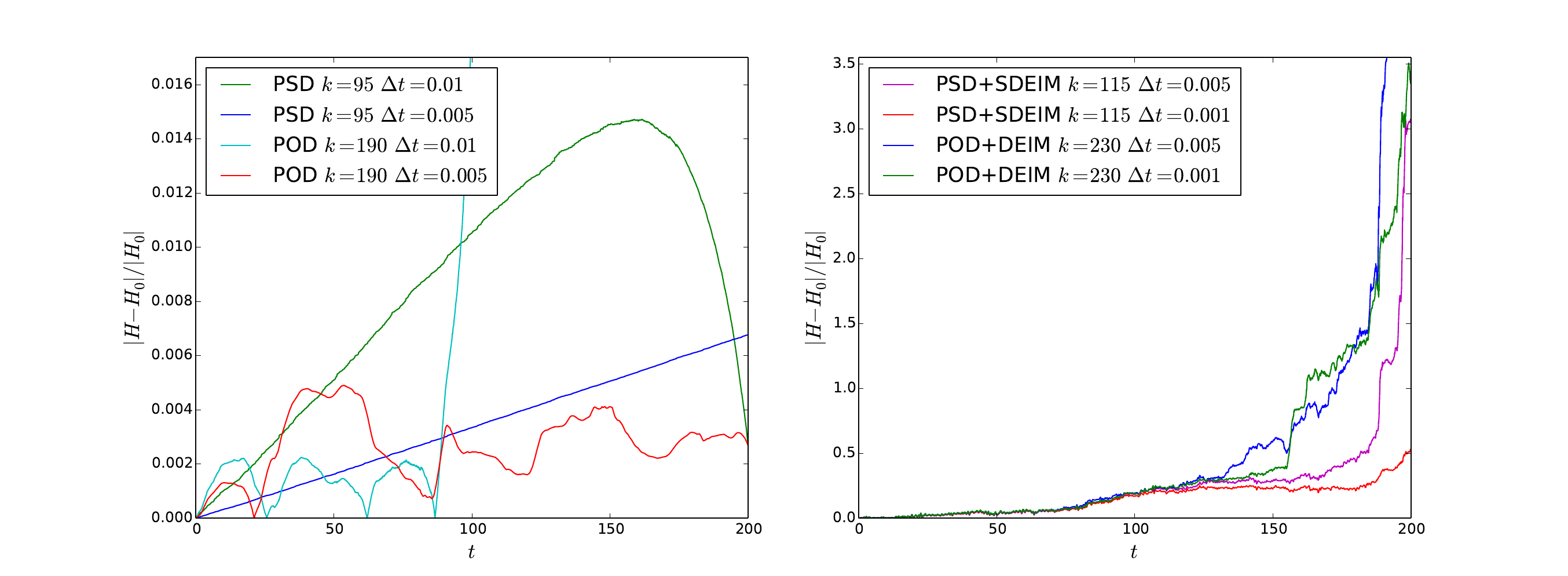}
		\caption{ The relative error of the Hamiltonian as a function of time for several example reduced model simulations of the stochastic NLS equation using the stochastic $R2$ method is depicted, where $H_0$ denotes the initial value of the Hamiltonian. Note that the Hamiltonian for the POD method with $k=190$ and $\Delta t = 0.01$ blows up.  }
		\label{fig: NLS Energy---R2}
\end{figure}

The advantage of maintaining the Hamiltonian structure in constructing a reduced model is evident even when explicit non-symplectic schemes are used to integrate the reduced equations. The simulations using the stochastic $R2$ method were carried out with the time steps $\Delta t = 0.01$, $\Delta t = 0.005$, and $\Delta t = 0.001$. Generally, shorter time steps were needed in order to maintain the stability of the numerical solution over the integration time. The numerical solution for $|\psi(x,t)|$ obtained from the reduced models at times $t=100$ and $t=200$ is compared against the reference solution in Figure~\ref{fig: NLS Solution---R2}. One can see that the PSD method with the time step $\Delta t = 0.01$ was able to produce an accurate solution up to the final time $t=200$, while the POD method with the same time step eventually became unstable. The relative error $e_1(t)$ is depicted in Figure~\ref{fig: NLS Instant Error---R2} and the relative error of the Hamiltonian is shown in Figure~\ref{fig: NLS Energy---R2}. The $R2$ method is not symplectic, therefore good preservation of the Hamiltonian is not expected, but also in this case we have observed that the PSD models produce more accurate solutions than the POD models of the same dimension and integrated with the same time step. The simulations using the stochastic Heun method yielded nearly identical results as the $R2$ method, therefore for brevity and clarity we skip presenting separate figures.

\subsection{Kubo oscillator}
\label{sec: Kubo oscillator}

The Kubo oscillator is a stochastic Hamiltonian system driven by a one-dimensional Wiener process  with the Hamiltonians given by $H(q,p)=p^2/2+q^2/2$ and $h(q,p)=\beta(p^2/2+q^2/2)$, where $\beta$ is the noise intensity (see \cite{MilsteinRepin}), and since the noise is one-dimensional, for simplicity we write $h\equiv h_1$.  It is an example of an oscillator with a fluctuating frequency and it was first introduced in the context of the line-shape theory (see \cite{Anderson1954}, \cite{Kubo1954}), but later also found many other applications in connection with mechanical systems, turbulence, laser theory, wave propagation (see \cite{VanKampen1976} and the references therein), magnetic resonance spectroscopy, nonlinear spectroscopy (see \cite{MukamelBook1995} and the references therein), single molecule spectroscopy (\cite{Jung2003}), and stochastic resonance (\cite{ChaudhuriMicroscopic2009}, \cite{Chaudhuri2010}, \cite{ChaudhuriNonequilibrium2009}, \cite{Gitterman2004}). The Kubo oscillator serves as a prototype for multiplicative stochastic processes, and since its solutions can be calculated analytically, it is often used for validation of numerical algorithms (see, e.g., \cite{Fox1987}, \cite{MaDing2015}, \cite{MilsteinRepin}, \cite{SunWang2016}). It is straightforward to verify that the exact solution is given by  

\begin{equation}
\label{eq:Kubo oscillator---exact solution}
q_e(t)=p_0 \sin(t+\beta W(t)) + q_0 \cos(t+\beta W(t)), \qquad\quad p_e(t)=p_0 \cos(t+\beta W(t)) - q_0 \sin(t+\beta W(t)),
\end{equation}

\noindent
where $q_0$ and $p_0$ are the initial conditions. Note that \eqref{eq:Kubo oscillator---exact solution} is the solution of the deterministic harmonic oscillator equation with the time argument shifted by $\beta W(t)$. It is also clearly evident that the Hamiltonian $H$ is preserved along the solution \eqref{eq:Kubo oscillator---exact solution}. Since $W(t)\sim N(0,t)$ is normally distributed, one can explicitly calculate the mean position and momentum as, respectively,

\begin{equation}
\label{eq:Kubo oscillator---exact mean solution}
\mathbb{E}[q_e(t)]=e^{-\frac{\beta^2}{2} t}(p_0 \sin t + q_0 \cos t), \qquad\quad \mathbb{E}[p_e(t)]=e^{-\frac{\beta^2}{2} t}(p_0 \cos t - q_0 \sin t).
\end{equation}

\subsubsection{Empirical data}
\label{sec: Empirical data for Kubo}

Suppose we have the following computational problem: we would like to calculate the expected value of the solution to \eqref{eq: Stochastic Hamiltonian system - general} for a large number of parameters $\beta$. In order to accurately approximate the mean value of a stochastic process one typically needs thousands, or even millions Monte Carlo runs. This presents a computational challenge which can be alleviated by model reduction. One can carry out full-scale computations only for a selected number of values of $\beta$, and the resulting data can be used to identify a reduced model, as described in Section~\ref{sec: Reduction of the number of Monte Carlo runs}. The lower dimensional equations can then be solved for other values of $\beta$. In our experiment we have considered the initial conditions $q_0=0$ and $p_0=1$, and we have used the exact solution \eqref{eq:Kubo oscillator---exact solution} to generate the empirical data for the following six (arbitrarily selected)  values of the parameter $\beta$:

\begin{equation}
\label{eq: beta values for Kubo}
\beta = \; 0.00095, \; 0.00097, \; 0.00099, \; 0.00101, \; 0.00103, \; 0.00105.
\end{equation}

\begin{figure}[tbp]
	\centering
		\includegraphics[width=.8\textwidth]{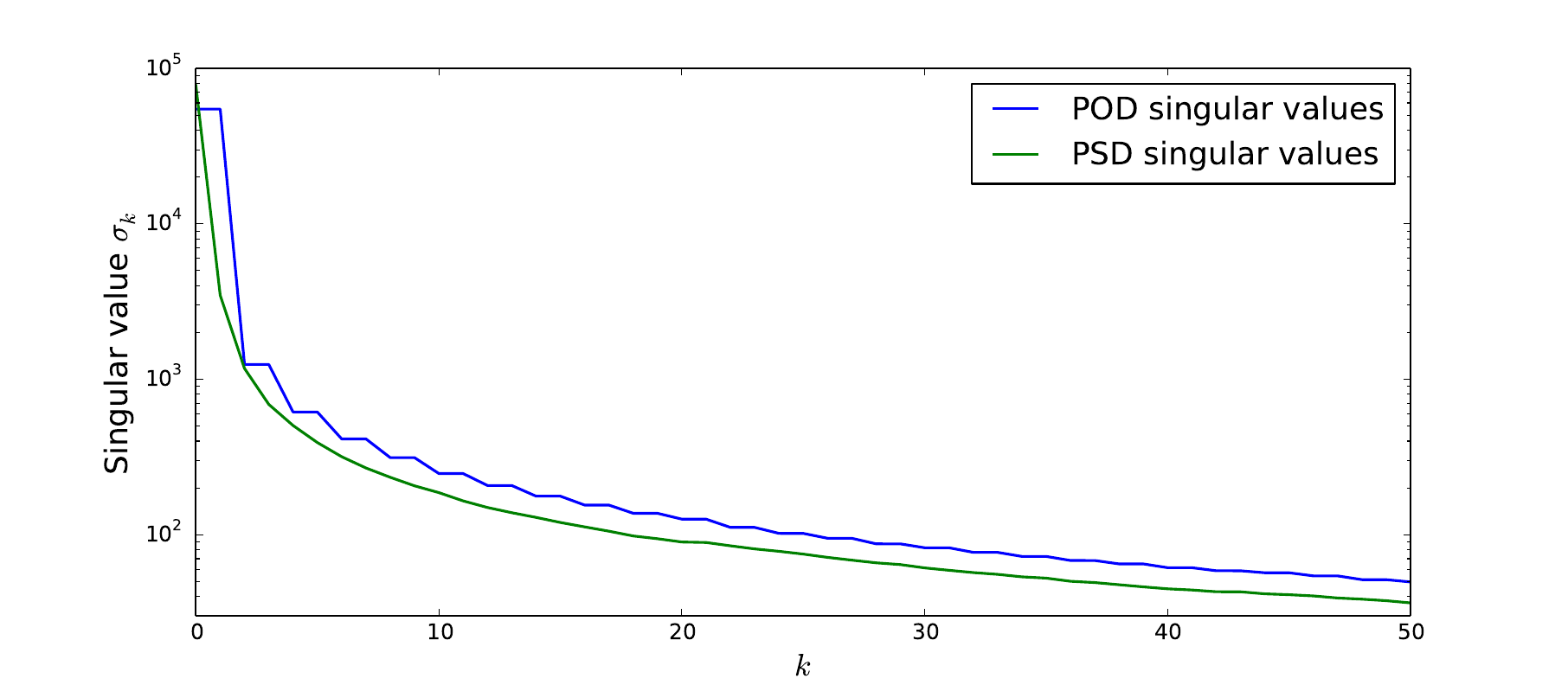}
		\caption{ The decay of the singular values for the POD and PSD reductions for the empirical data ensemble for the Kubo oscillator. Because of their large number and rapid decay, for clarity only the first 50 singular values are depicted. }
		\label{fig: Singular values for Kubo}
\end{figure}

\noindent
The exact solution \eqref{eq:Kubo oscillator---exact solution} was sampled over the time interval $0\leq t \leq 5000$ with the time step $\Delta t = 0.05$ for $M=10000$ independent realizations of the Wiener process, $W(t;\omega_1),\ldots,W(t;\omega_M)$, where $\omega_\nu$ denote elementary events in the probability space. Of course in the general situation the empirical data are generated using high-fidelity numerical methods applied to the full model, like in Section~\ref{sec: Stochastic Nonlinear Schrodinger Equation}, but our experiment was meant as a validation of the model reduction method, therefore the exact solution was used for convenience. Then, following the description of each algorithm in Section~\ref{sec: Reduction of the number of Monte Carlo runs}, reduced models were derived. The single realization of the $M$-dimensional Wiener process in \eqref{eq: System of SDEs for Monte Carlo} and \eqref{eq: System of stochastic Hamiltonian systems for Monte Carlo} was defined as $W^\nu(t)=W(t;\omega_\nu)$ for $\nu=1,\ldots,M$. Note that for the Kubo oscillator the drift and diffusion terms are linear, therefore no DEIM approximations were necessary. The decay of the singular values for the POD and PSD methods is depicted in Figure~\ref{fig: Singular values for Kubo}.

\subsubsection{Reduced model simulations}
\label{sec: Reduced model simulations for Kubo}

The reduced models obtained in Section~\ref{sec: Empirical data for Kubo} from the empirical data can now be solved for any other desired value of the parameter $\beta$.  To test the accuracy of the considered model reduction methods, we have compared the results of the reduced model simulations to a full-scale reference solution for a single arbitrary choice of the parameter.  The reference solution for $\beta=0.001$ was generated on the time interval $0 \leq t \leq 5000$ in the same way as the empirical data in Section~\ref{sec: Empirical data for Kubo}. The POD reduced model was solved using the stochastic R2 and Heun methods for $k=42$ (thus reducing the dimensionality of the system \eqref{eq: System of SDEs for Monte Carlo} from $2M=20000$ to 42). The PSD reduced model was solved using the stochastic St\"{o}rmer-Verlet method \eqref{eq:Stochastic Stormer-Verlet method} for $k=21$ (thus reducing the dimensionality of the system \eqref{eq: System of stochastic Hamiltonian systems for Monte Carlo} from $2M=20000$ to $2k=42$). Note that the Hamiltonians for the Kubo oscillator are separable, therefore the St\"{o}rmer-Verlet method is in this case explicit.

\begin{figure}[tbp]
	\centering
		\includegraphics[width=\textwidth]{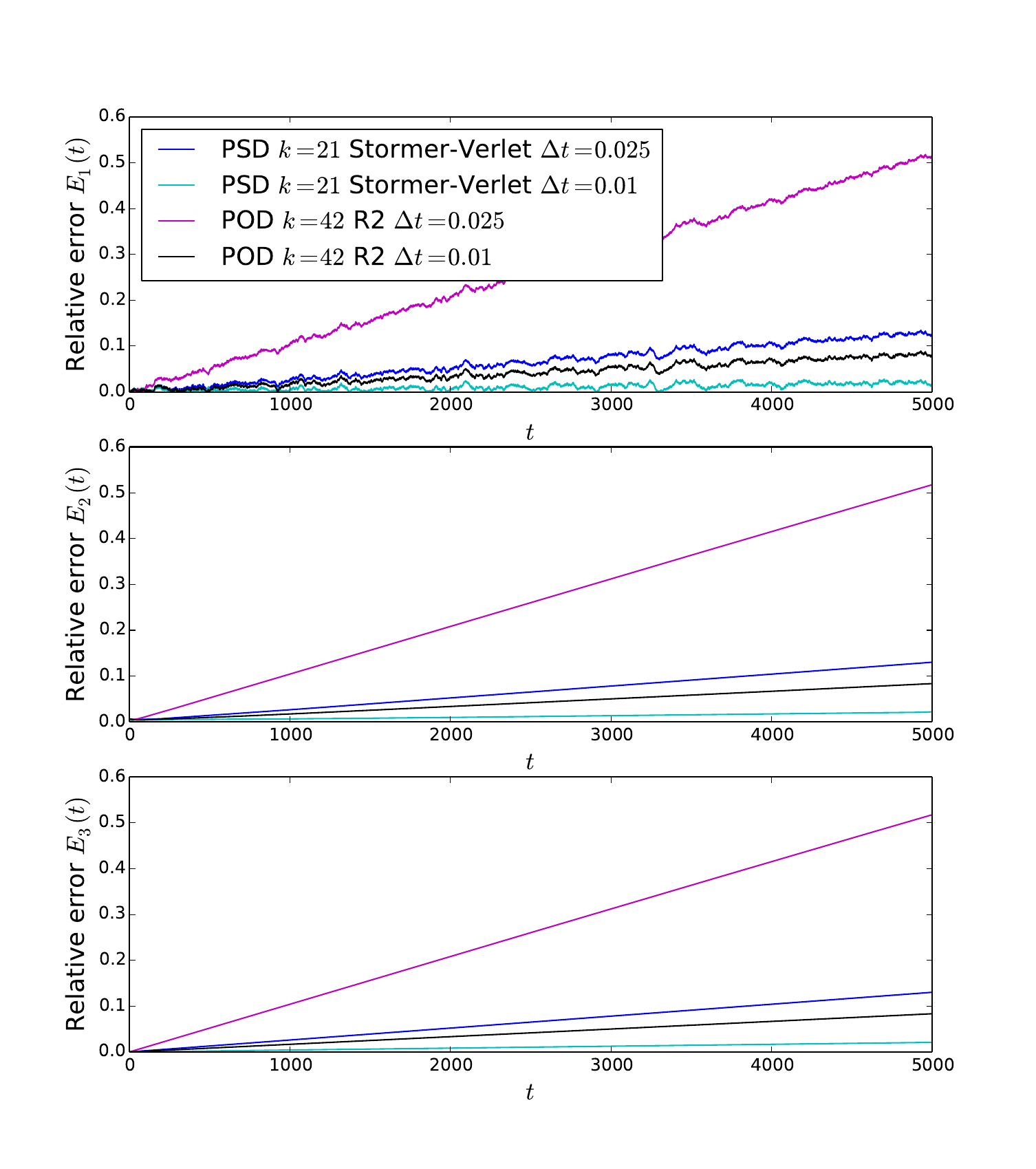} 
		\caption{ The relative errors $E_1(t)$ (\emph{Top}), $E_2(t)$ (\emph{Middle}), and $E_3(t)$ (\emph{Bottom}) for the reduced model simulations of the Kubo oscillator are depicted as functions of time. The PSD method combined with the stochastic St\"{o}rmer-Verlet scheme yields more accurate solutions than the POD method combined with the $R2$ integrator and using the same time step. The POD simulations with the Heun method yield nearly identical results as the $R2$ method, therefore separate plots are omitted for clarity. }
		\label{fig: Solution errors for Kubo}
\end{figure}

\begin{figure}[tbp]
	\centering
		\includegraphics[width=\textwidth]{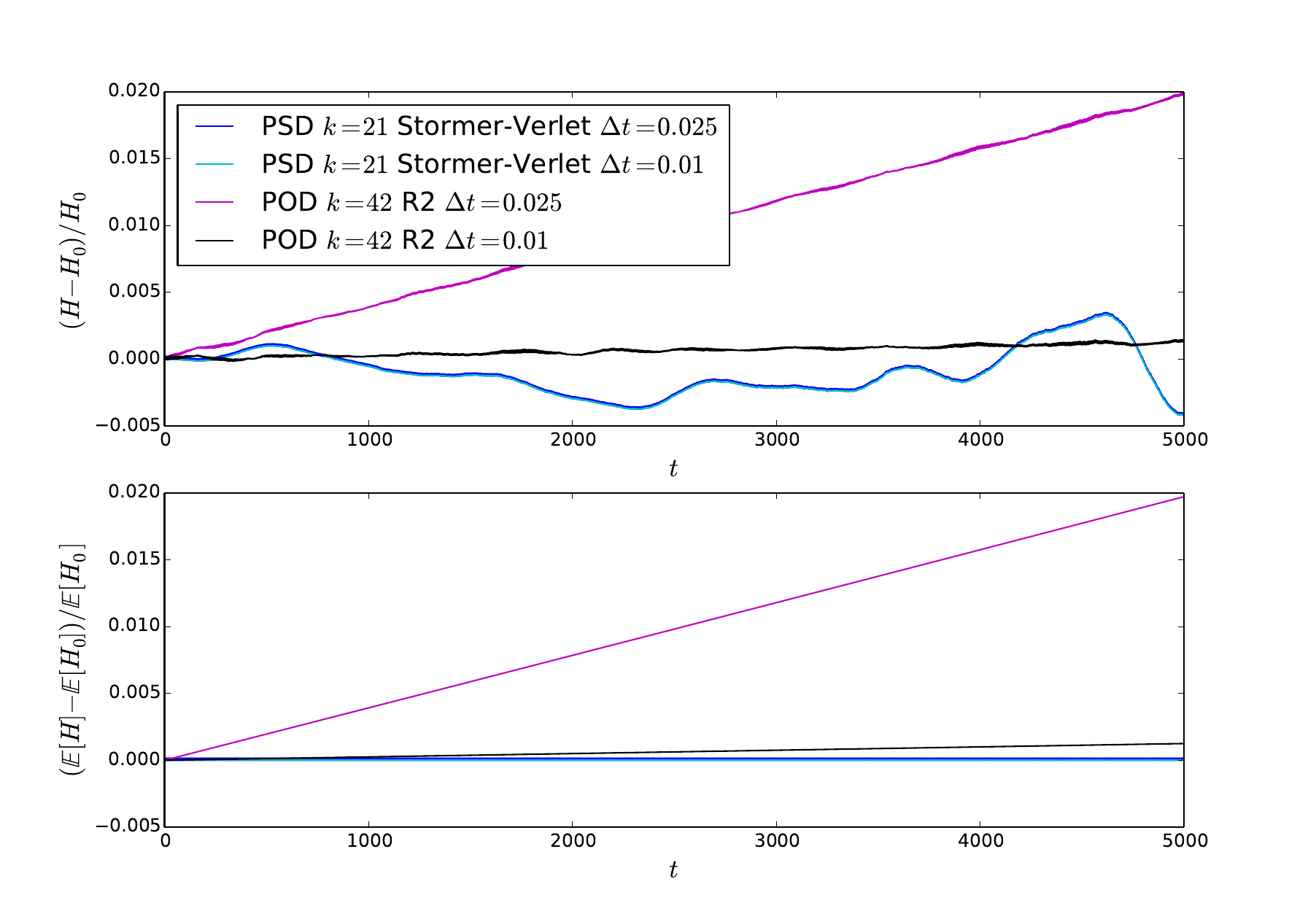} 
		\caption{ The relative errors of the Hamiltonian $H$ for a selected sample path (\emph{Top}) and of the mean Hamiltonian $\mathbb{E}[H]$ (\emph{Bottom}) for the reduced model simulations of the Kubo oscillator are depicted as functions of time, where $H_0$ denotes the initial value of the Hamiltonian. The PSD simulations using the St\"{o}rmer-Verlet method nearly exactly preserve the mean Hamiltonian, while the POD simulations using the $R2$ and Heun methods demonstrate a linear growth trend. Note that the plots for the PSD simulations with the time steps $\Delta t =0.025$ and $\Delta t = 0.01$ overlap very closely and are therefore barely distinguishable. The POD simulations with the Heun method yield nearly identical results as the $R2$ method, therefore separate plots are omitted for clarity. }
		\label{fig: Energy errors for Kubo}
\end{figure}

As a measure of accuracy of the reduced models we considered three types of error. First, we investigated the relative error of a selected single sample path, that is,

\begin{equation}
\label{eq: Relative error for a single path}
E_1(t) = \frac{\| u(t)-u_e(t) \|}{\| u_e(t)\|},
\end{equation}

\noindent
where $u_e=(q_e,p_e)$ and $u=(q,p)$ denote single sample paths of the exact reference solution and of the solution reconstructed from the reduced model, respectively, and $\|\cdot\|$ is the Euclidean norm on~$\mathbb{R}^2$. Second, we studied the relative mean-square error, that is,

\begin{equation}
\label{eq: Relative mean-square error}
E_2(t) = \sqrt{\frac{\mathbb{E}\big[ \| u(t)-u_e(t) \|^2 \big]}{\mathbb{E}\big[\| u_e(t)\|^2\big]}}.
\end{equation}

\noindent
Finally, we used the relative error of the mean (also known as the weak error), that is,

\begin{equation}
\label{eq: Relative weak error}
E_3(t) = \frac{\Big\| \mathbb{E}[u(t)]-\mathbb{E}[u_e(t)] \Big\|}{\Big\| \mathbb{E}[u_e(t)]\Big\|}.
\end{equation}

\noindent
All errors are depicted in Figure~\ref{fig: Solution errors for Kubo}. We can see that the reduced model simulations indeed yield good approximations of the exact solution. We also see that the PSD method combined with the stochastic St\"{o}rmer-Verlet scheme yields more accurate solutions than the POD method combined with the $R2$ and Heun integrators. The numerical values of the Hamiltonian $H$ for a selected single sample path and of the mean Hamiltonian $\mathbb{E}[H]$ as functions of time are depicted in Figure~\ref{fig: Energy errors for Kubo}. The PSD simulations using the St\"{o}rmer-Verlet method nearly exactly preserve the mean Hamiltonian, while the POD simulations using the $R2$ and Heun methods demonstrate a linear growth trend.

\subsection{Forced Kubo oscillator}
\label{sec: Forced Kubo oscillator}

In this experiment we consider the Kubo oscillator described, as in Section~\ref{sec: Kubo oscillator}, by the Hamiltonians $H(q,p)=p^2/2+q^2/2$ and $h(q,p)=\beta(p^2/2+q^2/2)$, subject to external damping, with the forcing terms given by $F(q,p)=-\nu p$ and $f(q,p)=-\beta \nu p$, where $\nu$ is the damping coefficient. Since the noise is one-dimensional, for simplicity we write $h\equiv h_1$ and $f\equiv f_1$. It is straightforward to verify that the exact solution is given by  

\begin{align}
\label{eq:Damped Kubo oscillator---exact solution}
q_e(t)&= q_0 e^{-\frac{\nu}{2}(t+\beta W(t))} \cos \omega (t+\beta W(t)) + \frac{1}{\omega}(p_0+\frac{\nu}{2} q_0) e^{-\frac{\nu}{2}(t+\beta W(t))} \sin \omega (t+\beta W(t)), \nonumber \\
p_e(t)&= p_0 e^{-\frac{\nu}{2}(t+\beta W(t))} \cos \omega (t+\beta W(t)) - \frac{1}{\omega}(q_0+\frac{\nu}{2} p_0) e^{-\frac{\nu}{2}(t+\beta W(t))} \sin \omega (t+\beta W(t)),
\end{align}

\noindent
where $q_0$ and $p_0$ are the initial conditions, the angular frequency is $\omega=\frac{1}{2}\sqrt{4-\nu^2}$, and we have assumed the underdamped case $0\leq \nu < 2$. Note that \eqref{eq:Damped Kubo oscillator---exact solution} is the solution of the deterministic damped harmonic oscillator with the time argument shifted by $\beta W(t)$. Given that $W(t)\sim N(0,t)$ is normally distributed, one can explicitly calculate the expected value of the Hamiltonian $H$ as a function of time as

\begin{align}
\label{eq:The expected value of the Hamiltonian for the Kubo oscillator}
E\Big(H\big(q_e(t),p_e(t)\big)\Big)=a e^{-\frac{\nu (2-\beta^2 \nu)}{2}t} + e^{-((2-\nu^2)\beta^2+\nu)t}\Big[ b \cos\big(2 (1-\beta^2 \nu)\omega t \big) + c \sin\big(2 (1-\beta^2 \nu)\omega t \big) \Big],
\end{align}

\noindent
where

\begin{align}
\label{eq:Coefficients in the formula for E(H)}
a=\frac{2 (p_0^2 + q_0^2 + \nu p_0 q_0)}{4 - \nu^2}, \qquad\qquad b=-\frac{\nu^2 (p_0^2 + q_0^2)+4 \nu p_0 q_0}{2 (4 - \nu^2)}, \qquad\qquad c=\frac{\nu (q_0^2-p_0^2)}{2 \sqrt{4 - \nu^2}}.
\end{align}

\subsubsection{Empirical data}
\label{sec: Empirical data for forced Kubo}

Suppose that, similar to Section~\ref{sec: Empirical data for Kubo}, we have the following computational problem: a solution (or its statistical properties) to \eqref{eq: Stochastic dissipative Hamiltonian system} is needed for a large number of parameters $\beta$ and $\nu$. In order to efficiently calculate, e.g., the expected value of the solution, we use full-model data for a selected number of values of $\beta$ and $\nu$ to construct a reduced model which is less expensive to solve. In our experiment we have considered the initial conditions $q_0=2$ and $p_0=0$, and we have used the exact solution \eqref{eq:Damped Kubo oscillator---exact solution} to generate the empirical data for $\beta=0.001$ and for the following six (arbitrarily selected) values of the damping coefficient $\nu$:

\begin{equation}
\label{eq: nu values for forced Kubo}
\nu = \; 0.00097, \; 0.00099, \; 0.00101, \; 0.00103, \; 0.00105, \; 0.00107.
\end{equation}

\begin{figure}[tbp]
	\centering
		\includegraphics[width=.8\textwidth]{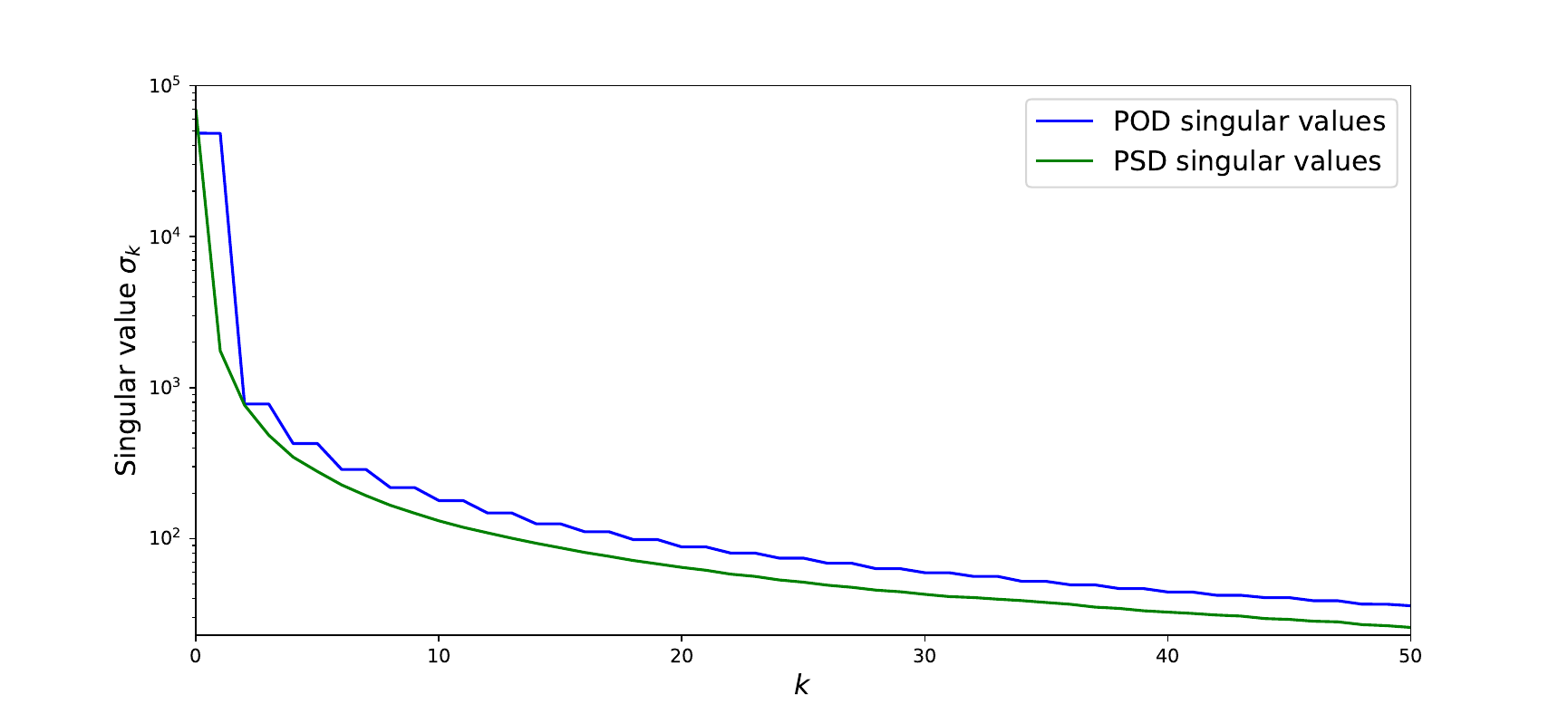}
		\caption{ The decay of the singular values for the POD and PSD reductions for the empirical data ensemble for the forced Kubo oscillator. Because of their large number and rapid decay, for clarity only the first 50 singular values are depicted. }
		\label{fig: Singular values for forced Kubo}
\end{figure}

\noindent
The exact solution \eqref{eq:Damped Kubo oscillator---exact solution} was sampled over the time interval $0\leq t \leq 5000$ with the time step $\Delta t = 0.05$ for $M=10000$ independent realizations of the Wiener process, $W(t;\omega_1),\ldots,W(t;\omega_M)$, where $\omega_\nu$ denote elementary events in the probability space. Then, following the description of each algorithm in Section~\ref{sec: Reduction of the number of Monte Carlo runs}, reduced models were derived. The single realization of the $M$-dimensional Wiener process in \eqref{eq: System of SDEs for Monte Carlo} and \eqref{eq: System of stochastic Hamiltonian systems for Monte Carlo} was defined as $W^\nu(t)=W(t;\omega_\nu)$ for $\nu=1,\ldots,M$. Note that for the forced Kubo oscillator with the linear forcing terms, the drift and diffusion terms are linear, therefore no DEIM approximations were necessary. The decay of the singular values for the POD and PSD methods is depicted in Figure~\ref{fig: Singular values for forced Kubo}.

\subsubsection{Reduced model simulations}
\label{sec: Reduced model simulations for forced Kubo}

The reduced models obtained in Section~\ref{sec: Empirical data for forced Kubo} can now be solved for any other desired value of the parameter $\nu$. To test the accuracy of the considered model reduction methods, we have compared the results of the reduced model simulations to a full-scale reference solution for a single arbitrary choice of the parameter. The reference solution for $\beta=0.001$ and $\nu=0.001$ was generated on the time interval $0 \leq t \leq 5000$ in the same way as the empirical data in Section~\ref{sec: Empirical data for forced Kubo}. Both the POD and PSD reduced models were integrated using the stochastic St\"{o}rmer-Verlet method. Note that in the case of the forced Kubo oscillator, the St\"{o}rmer-Verlet method is implicit for both the POD and PSD models. As explained in Section~\ref{sec: PSD time integration}, when applied to the PSD reduced model, the St\"{o}rmer-Verlet method is a structure-preserving Lagrange-d'Alembert integrator, whereas it does not have this property in the case of the POD reduced model. The POD reduced model was solved for $k=42$ (thus reducing the dimensionality of the full system from $2M=20000$ to 42). The PSD reduced model was solved for $k=21$ (thus reducing the dimensionality of the full system from $2M=20000$ to $2k=42$).

\begin{figure}[tbp]
	\centering
		\includegraphics[width=\textwidth]{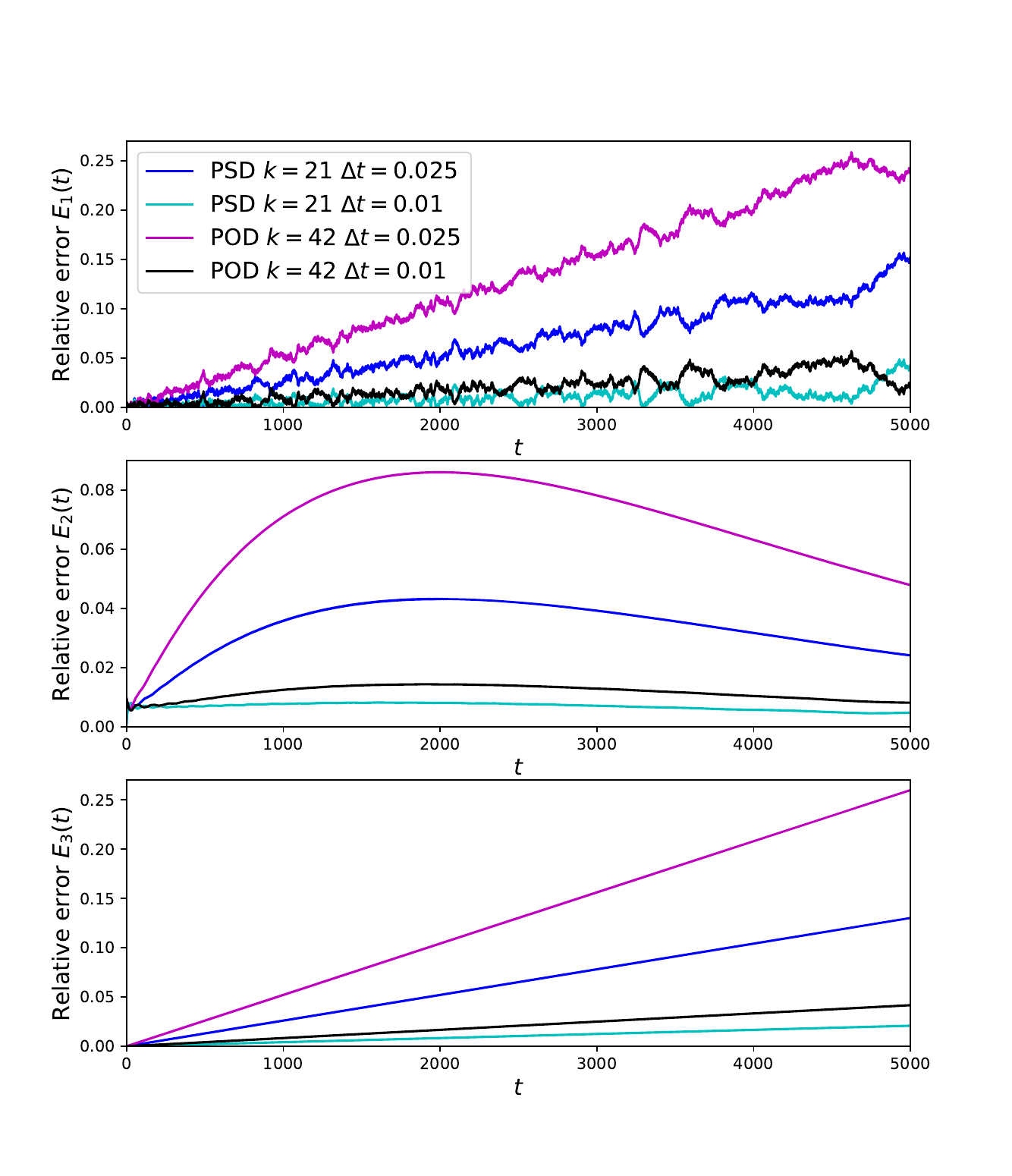} 
		\caption{ The relative errors $E_1(t)$ (\emph{Top}), $E_2(t)$ (\emph{Middle}), and $E_3(t)$ (\emph{Bottom}) for the reduced model simulations of the forced Kubo oscillator using the stochastic St\"{o}rmer-Verlet scheme are depicted as functions of time. The PSD method yields more accurate solutions than the POD method using the same time step.}
		\label{fig: Solution errors for forced Kubo}
\end{figure}

\begin{figure}[tbp]
	\centering
		\includegraphics[width=0.82\textwidth]{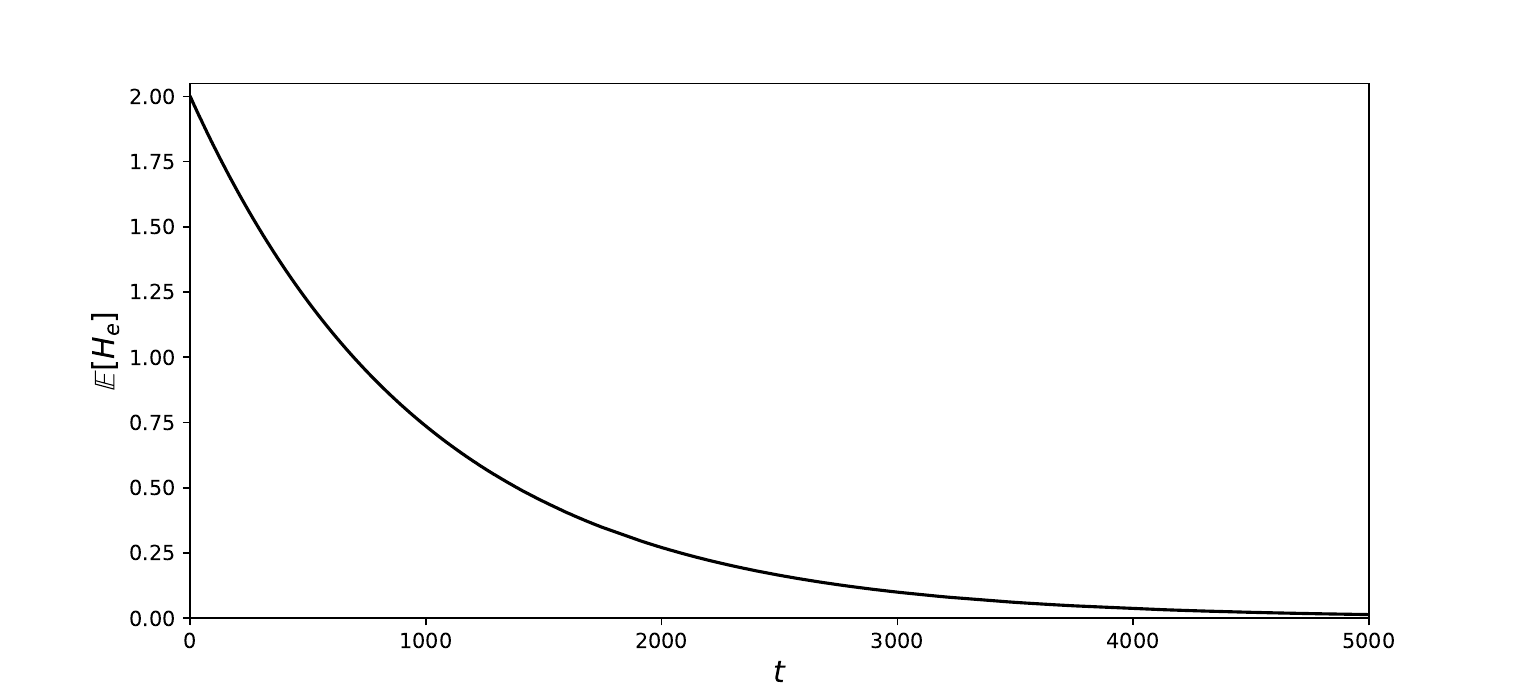} 
		\caption{ The time dependence of the expected value of the Hamiltonian $\mathbb{E}[H_e]$ for the exact solution~\eqref{eq:Damped Kubo oscillator---exact solution} of the forced Kubo oscillator with the initial conditions $q_0=2$ and $p_0=0$, and the parameters $\beta=0.001$ and $\nu=0.001$.}
		\label{fig: Exact energy for forced Kubo}
\end{figure}

\begin{figure}[tbp]
	\centering
		\includegraphics[width=0.85\textwidth]{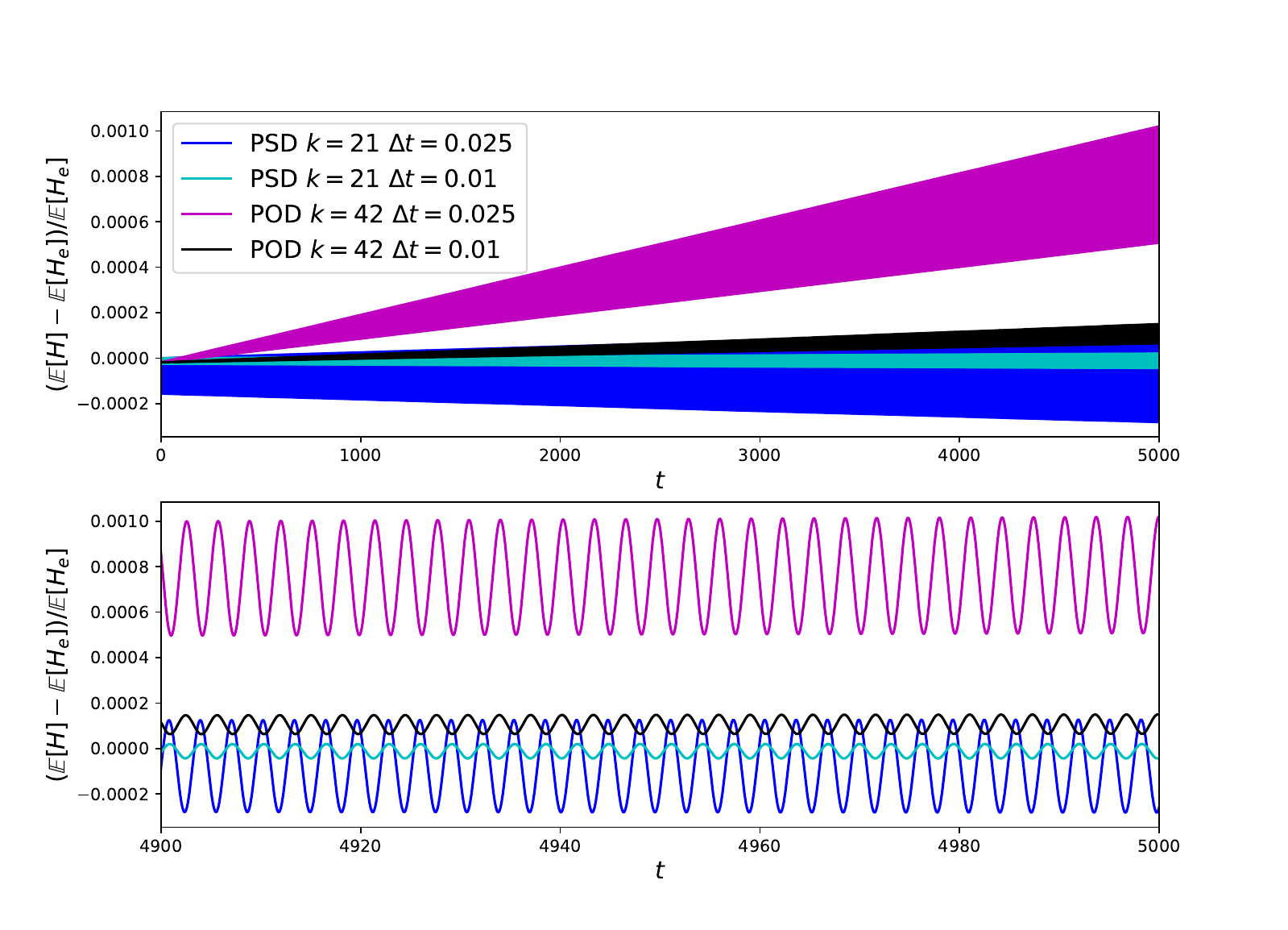} 
		\caption{ \emph{Top:} The relative error of the mean Hamiltonian $\mathbb{E}[H]$ for the reduced model simulations of the forced Kubo oscillator using the stochastic St\"{o}rmer-Verlet method is depicted as a function of time over the whole integration interval, where $H_e=H(q_e(t),p_e(t))$ denotes the value of the Hamiltonian for the exact solution~\eqref{eq:Damped Kubo oscillator---exact solution}. It is evident that the PSD simulations capture the evolution of Hamiltonian more accurately than the POD simulations with the same time step. \emph{Bottom:} The same plot zoomed in on the time interval [4900,5000] for a more detailed depiction of the oscillatory character of the time dependence of the error. }
		\label{fig: Energy errors for forced Kubo}
\end{figure}

The errors $E_1(t)$, $E_2(t)$, and $E_3(t)$ (see \eqref{eq: Relative error for a single path}, \eqref{eq: Relative mean-square error}, and \eqref{eq: Relative weak error}, respectively) are depicted in Figure~\ref{fig: Solution errors for forced Kubo}. We can see that the reduced model simulations indeed yield good approximations of the exact solution. We also see that the PSD method yields more accurate solutions than the POD method with the same time step. The theoretical time evolution \eqref{eq:The expected value of the Hamiltonian for the Kubo oscillator} of the mean Hamiltonian $\mathbb{E}[H_e]$ for the chosen initial conditions and parameters is depicted in Figure~\ref{fig: Exact energy for forced Kubo}, whereas the errors of the corresponding numerical values $\mathbb{E}[H]$ are depicted in Figure~\ref{fig: Energy errors for forced Kubo}. It is evident that the PSD simulations capture the evolution of the Hamiltonian more accurately than the POD simulations with the same time step.

\section{Summary}
\label{sec: Summary}

We have successfully demonstrated that SVD-based model reduction methods known for ordinary differential equations can be extended to stochastic differential equations, and can be used to reduce the computational cost arising from both the high dimension of the considered stochastic system and the large number of independent Monte Carlo runs. We have also argued that in model reduction of stochastic Hamiltonian systems it is advisable to maintain their symplectic or variational structures, to which end we have adapted to the stochastic setting the proper symplectic decomposition method known for deterministic Hamiltonian systems. We have further applied our proposed techniques to a semi-discretization of the stochastic Nonlinear Schr\"{o}dinger equation and to the Kubo oscillator, providing numerical evidence that model reduction is a viable tool for obtaining accurate numerical approximations of stochastic systems, and that preserving the geometric structures of stochastic Hamiltonian systems results in more accurate and stable solutions that conserve energy better than when the non-geometric approach is used.

Our work can be extended in many ways. A natural follow-up study would be the application of the presented model reduction techniques to semi-discretizations of other stochastic PDEs with underlying geometric structures, such as the stochastic Camassa-Holm equation (\cite{BendallCotterHolm2021}, \cite{HolmTyranowskiSolitons}, \cite{HolmTyranowskiVirasoro}), the stochastic quasi-geostrophic equation (\cite{BendallCotter2019}, \cite{CotterCrisanHolm2020}, \cite{CotterHolm2020}, \cite{CotterHolm2019}, \cite{CotterHolm2017}), the stochastic Korteweg-de Vries equation (\cite{ChenWang2005}, \cite{BouardDebussche1998}, \cite{HermanRose2009}, \cite{HolmTyranowskiVirasoro}), or the stochastic Sine-Gordon equation (\cite{HairerShen2016}, \cite{Tuckwell2016}, \cite{TyranowskiDesbrunRAMVI}, \cite{WangChen2021}) to name just a few. Particle discretizations of collisional Vlasov equations have been recently proved to have the structure of stochastic forced Hamiltonian systems (\cite{KrausTyranowski2019}, \cite{TyranowskiVlasovMaxwell}, \cite{TyranowskiKraus2021}), therefore kinetic plasma theory is yet another area where structure-preserving model reduction techniques could be applied. Furthermore, it would be interesting to extend the methods discussed in our work to stochastic non-canonical systems. For instance, one could consider structure-preserving model reduction techniques for stochastic Poisson systems by combining the method developed in \cite{Hesthaven2020} with the stochastic geometric integrators presented in \cite{Cohen2014} and \cite{CohenVilmart2021}. This would be of great interest for systems appearing in gyrokinetic and guiding-center theories (see \cite{Balescu1994}, \cite{BottinoSonnendrucker2015}, \cite{Brizard2000}, \cite{BrizardTronci2016}, \cite{CaryBrizard2009}, \cite{Pfirsch1984}, \cite{PfirschMorrison1985}, \cite{SugamaGyrokinetic2000}, \cite{TyranowskiDesbrunLinearLagrangians}, \cite{Eijnden1998}).

\section*{Acknowledgements}

We would like to thank Tobias Blickhan, Michael Kraus, Paul Skerritt and Udo von Toussaint for useful comments and references. The study is a contribution to the Reduced Complexity Models grant number ZT-I-0010 funded by the Helmholtz Association of German Research Centers.

\bibliographystyle{abbrv}

\end{document}